\documentclass[12pt]{amsart}

\usepackage{a4wide}

\usepackage{enumerate}
\usepackage{xcolor}
\usepackage{amssymb} 
\usepackage{tikz}
\usepackage{array,multirow}
\usetikzlibrary{arrows.meta, positioning, calc}
\newcommand{\circleinnersep}{3}
\newcommand{\nodedistance}{1.2}

\newtheorem{theorem}{Theorem}[section]
\newtheorem{lemma}[theorem]{Lemma}
\newtheorem{proposition}[theorem]{Proposition}

\newtheorem{theo}{Theorem}

\theoremstyle{definition}
\newtheorem{remark}[theorem]{Remark}
\newtheorem{example}[theorem]{Example}

\newcommand{\C}{\ensuremath{\mathbb{C}}}
\newcommand{\R}{\ensuremath{\mathbb{R}}}
\renewcommand{\H}{\ensuremath{\mathbb{H}}}
\newcommand{\g}[1]{\ensuremath{\mathfrak{#1}}}
\newcommand{\s}[1]{\ensuremath{\mathsf{#1}}}
\newcommand{\cal}[1]{\ensuremath{\mathcal{#1}}}

\DeclareMathOperator{\codim}{codim}
\DeclareMathOperator{\Ad}{Ad}
\DeclareMathOperator{\ad}{ad}
\DeclareMathOperator{\id}{id}

\DeclareMathOperator{\Exp}{Exp}
\DeclareMathOperator{\spann}{span}

\begin{document}
\title[Cohomogeneity one actions on symmetric spaces of higher rank]{Cohomogeneity one actions on  symmetric spaces\\of noncompact type and higher rank}

\author[J.\ C.\ D\'{\i}az-Ramos]{Jos\'{e} Carlos D\'{\i}az-Ramos}
\author[M.\ Dom\'{\i}nguez-V\'{a}zquez]{Miguel Dom\'{\i}nguez-V\'{a}zquez}
\author[T.\ Otero]{Tom\'as Otero}

\address{CITMAga, 15782 Santiago de Compostela, Spain.\newline\indent Department of Mathematics, Universidade de Santiago de Compostela, Spain}
\email{josecarlos.diaz@usc.es}
\email{miguel.dominguez@usc.es}
\email{tomas.otero.casal@usc.es}

%\date{}

\begin{abstract}
We develop a new structural result for cohomogeneity one actions on (not necessarily irreducible) symmetric spaces of noncompact type and arbitrary rank. We apply this result to classify cohomogeneity one actions on $\s{SL}_{n}(\R)/\s{SO}_n$, $n\geq 2$, up to orbit equivalence. We also reduce the classification problem on a reducible space to the classification on each one of its irreducible factors, which in particular allows to classify cohomogeneity one actions on any finite product of hyperbolic spaces.
\end{abstract}

\thanks{The authors have been supported by the projects PID2019-105138GB-C21/AEI/10.13039/501100011033 (Spain) and ED431C 2019/10, ED431F 2020/04 (Xunta de Galicia, Spain). The second and third authors  acknowledge  support  of  the  Ram\'on y Cajal program (AEI, Spain) and the FPU program (Ministry of Universities, Spain), respectively. 
}



\maketitle

%    Text of article.
\section{Introduction and main results}
When studying proper isometric actions on a given Riemannian manifold, it is natural to investigate those that produce hypersurfaces as their regular orbits: these are the so-called cohomogeneity one actions. This kind of study makes special sense in manifolds with a large isometry group, as is the case of symmetric spaces. 

The classification of cohomogeneity one actions on irreducible symmetric spaces of compact type up to orbit equivalence was completed by Kollross in \cite{Kollross:tams}. However, neither the group theoretical approach used in this classification nor the use of duality of symmetric spaces allow to derive complete classification results in the noncompact case (see~\cite{BT:tohoku}, \cite{K:dual} for partial results). This is ultimately due to the fact that noncompact real semisimple Lie groups enjoy a much richer lattice of subgroups than compact Lie groups. This is why the development of specific techniques, based on the algebraic and geometric structure of symmetric spaces of noncompact type, has been shown to be crucial.

The particular but important case of rank one spaces (that is, the hyperbolic spaces over the normed division algebras) was first addressed in real hyperbolic spaces by Cartan~\cite{Cartan}, following a geometric approach. However, the classification in the other rank one spaces has only been concluded eighty years later, after Berndt and Tamaru's article~\cite{BT:tams} and the very recent classification for quaternionic hyperbolic spaces due to the first two authors and Rodr\'iguez-V\'azquez~\cite{DDR:crelle}, both works based on geometric, and especially, algebraic ideas. 

The higher rank setting seems to be even more complicated. In~\cite{BT:crelle}, Berndt and Tamaru proposed a general procedure for the classification of cohomogeneity one actions on irreducible symmetric spaces up to orbit equivalence. They proved (see~\S\ref{subsec:classes} for details) that any such action either induces a regular foliation, or has a totally geodesic singular orbit, or can be obtained by two new methods, called \emph{canonical extension} and \emph{nilpotent construction}. Cohomogeneity one actions inducing regular foliations on (possibly reducible) symmetric spaces had been previously classified into two families (which we will call of \emph{horospherical} or \emph{solvable type}) in~\cite{BT:jdg} and~\cite{BDT:jdg}, whereas those with a totally geodesic singular orbit (on irreducible spaces) had been determined in~\cite{BT:tohoku}. The canonical extension method allows to extend an action from certain totally geodesic submanifolds, called boundary components, to the whole ambient space. Since totally geodesic submanifolds of symmetric spaces are symmetric spaces of lower rank, this method suggests a rank reduction approach to the classification problem. However, a boundary component of an irreducible symmetric space might be reducible, and no general study of cohomogeneity one actions on reducible symmetric spaces of noncompact type has been developed yet, as far as the authors know. Finally, the actual application of the nilpotent construction in concrete spaces seems to be a very difficult task. Indeed, finding an effective application of this general method to the particular case of quaternionic hyperbolic spaces was the fundamental obstacle that delayed the achievement of the classification finally obtained in~\cite{DDR:crelle}.

Due to the difficulty in applying the nilpotent construction, the lack of a general theory for cohomogeneity one actions on reducible spaces, and the incomplete understanding of the interaction between the nilpotent construction and the canonical extension, only a few explicit classifications are known. Apart from the rank one setting, all known classifications correspond to certain irreducible symmetric spaces of rank two, namely the spaces $\s{SL}_3(\R)/\s{SO}_3$, $\s{SL}_3(\C)/\s{SU}_3$, $\s{SL}_3(\H)/\s{Sp}_3$, $\s{SO}_5(\C)/\s{SO}_5$, $\s{G}_2^2/\s{SO}_4$, $\s{G}_2^\C/\s{G}_2$, and the noncompact real and complex two-plane Grassmannians; see~\cite{BT:crelle}, \cite{BD:tg}, and the recent work~\cite{Solonenko} by Solonenko.

The aim of this article is twofold. On the one hand, we present a new structural result for cohomogeneity one actions on symmetric spaces of noncompact type, which provides an efficient tool to deal with spaces which are reducible or of rank higher than two. Indeed, for reducible spaces we show that the classification problem can be reduced to the corresponding problem on each one of the factors. On the other hand, we show the power of these results by deriving the first explicit classifications of cohomogeneity one actions on symmetric spaces of noncompact type and rank higher than two, namely on the spaces $\s{SL}_{n}(\R)/\s{SO}_n$ and on any finite product of rank~one spaces. 

We now state the structural result that we prove in this article. The subtle but impor\-tant improvement in relation to Berndt and Tamaru's result is that the only actions with a singular orbit that can be obtained by canonical extension, and not by nilpotent construction, are those extending an action of a maximal proper reductive subgroup (or equivalently, with a totally geodesic singular orbit) on a boundary component which is either irreducible or a product of two mutually homothetic symmetric spaces of rank one. See below for the explicit description of the different types of actions mentioned in the theorem.

\begin{theo}\label{th:classification}
	Let $M=G/K$ be a symmetric space of noncompact type, and let $H$ be a connected closed subgroup of $G$. Then $H$ acts on $M$  with cohomogeneity one if and only if the $H$-action is orbit equivalent to one of the following:
	\begin{enumerate}%[{\rm (1)}]
		\item[\rm (FH)] An action inducing a regular codimension one foliation of horospherical type. \label{item:FH}
		\item[\rm (FS)] An action  inducing a regular codimension one foliation of solvable type. 
		\item[\rm (CEI)] The canonical extension of a cohomogeneity one action with a totally geodesic singular orbit on an irreducible boundary component.
		\item[\rm (CER)] The canonical extension of a cohomogeneity one diagonal action on a reducible boundary component of rank two whose two factors are homothetic.
		\item[\rm (NC)] An action obtained by nilpotent construction.
	\end{enumerate}
\end{theo}

We will now introduce the necessary context and notation to explain each one of the types of actions appearing in Theorem~\ref{th:classification}. More information can be found in Section~\ref{sec:parabolic}.

Let $M=G/K$ be a Riemannian symmetric space of noncompact type, where $G$ is (a finite covering of) the connected component of the identity of the isometry group of $M$, and $K$ is the isotropy subgroup at some base point $o\in M$. Recall that $M$ is isometric to an open Euclidean ball endowed with a symmetric metric of nonpositive sectional curvature. The semisimple Lie algebra $\g{g}$ admits a Cartan decomposition $\g{g}=\g{k}\oplus\g{p}$, where $\g{p}$ can be identified with $T_oM$. We can define a positive definite inner product $\langle\cdot,\cdot \rangle$ on $\g{g}$ that agrees with the Killing form $\cal{B}$ of $\g{g}$ on $\g{p}$, with $-\cal{B}$ on $\g{k}$, and that makes $\g{k}$ and $\g{p}$ orthogonal. In this article,  $\ominus$ will denote orthogonal complement with respect to $\langle\cdot,\cdot\rangle$. A choice of a maximal abelian subspace  $\g{a}$ of $\g{p}$ determines a restricted root space decomposition $\g{g}=\g{g}_0\oplus(\bigoplus_{\lambda\in\Sigma}\g{g}_\lambda)$, where $\Sigma\subset\g{a}^*$ is the set of restricted roots. Let $\Sigma^+\subset\Sigma$ be a set of positive roots, and $\Lambda\subset \Sigma^+$ the associated set of simple roots. Recall that $|\Lambda|=\dim\g{a}=\mathrm{rank}\; M$. Then $\g{g}=\g{k}\oplus\g{a}\oplus\g{n}$ is an Iwasawa decomposition of $\g{g}$, where $\g{n}=\bigoplus_{\lambda\in\Sigma^+}\g{g}_\lambda$. 

The connected solvable Lie subgroup $AN$ of $G$ with Lie algebra $\g{a}\oplus\g{n}$ acts simply transitively on $M$. Hence, any connected subgroup of $G$ with Lie algebra of codimension one in $\g{a}\oplus\g{n}$ acts with cohomogeneity one and no singular orbits on $M$. As shown in~\cite{BDT:jdg} (or in~\cite{BT:jdg} for the case that $M$ is irreducible), any cohomogeneity one action on $M$ without singular orbit, or equivalently, inducing a \emph{regular foliation}, arises in this way, up to orbit equivalence. Moreover, these homogeneous regular foliations have been classified into two types, which correspond to items (FH) and (FS) in Theorem~\ref{th:classification}:
\begin{enumerate}
	\item[\rm (FH)] \emph{Horospherical type}, if it is induced by the action of the connected subgroup of $AN$ with Lie algebra $(\g{a}\ominus\ell)\oplus\g{n}$, for some one-dimensional subspace $\ell$ of $\g{a}$, up to orbit equivalence. Such regular foliations are characterized by the property that all their orbits are mutually congruent. For some choices of $\ell$, these orbits are horospheres (see~\cite[Remark~5.4]{DST:jmpa}).
	\item[\rm (FS)] \emph{Solvable type}, if it is induced by the action of the connected subgroup of $AN$ with Lie algebra $\g{a}\oplus(\g{n}\ominus\ell)$, for some one-dimensional subspace $\ell$ of a simple root space $\g{g}_\alpha$, $\alpha\in \Lambda$, up to orbit equivalence. These foliations have exactly one minimal leaf.
\end{enumerate}

We will now describe the remaining actions in Theorem~\ref{th:classification}. 
Let $\Phi\subset \Lambda$. Define $\Sigma_{\Phi}=\Sigma\cap\spann\Phi$, and $\Sigma^+_\Phi=\Sigma^+\cap \spann\Phi$. Consider the abelian Lie subalgebra $
\g{a}_\Phi=\bigcap_{\alpha\in\Phi}\ker\alpha$, and the nilpotent subalgebra $\g{n}_\Phi=\bigoplus_{\lambda\in \Sigma^+\setminus\Sigma^+_\Phi}\g{g}_\lambda\subset\g{n}$. The connected Lie subgroup $A_\Phi N_\Phi$ of $AN$ with Lie algebra $\g{a}_\Phi\oplus\g{n}_\Phi$ is known to act freely and polarly on $M$, as follows from the horospherical decomposition of $M$ associated with the parabolic subgroup of $G$ determined by $\Phi$ (see~\cite{DV:imnr}). This means that there is a totally geodesic submanifold $B_\Phi$ of $M$ passing through $o$ and intersecting each $A_\Phi N_\Phi$-orbit perpendicularly (and exactly once). Such totally geodesic submanifold $B_\Phi$, which is called a \emph{boundary component} in the context of the maximal Satake compactification of $M$, is intrinsically a symmetric space of noncompact type and rank $|\Phi|$. If we denote by $\g{b}_\Phi\cong T_o B_\Phi$ the Lie triple system corresponding to the totally geodesic submanifold $B_\Phi$, then $\g{s}_\Phi=[\g{b}_\Phi,\g{b}_\Phi]\oplus\g{b}_\Phi$ is a real semisimple Lie algebra, and the associated connected subgroup $S_\Phi$ of $G$ is (up to a finite covering) the connected component of the identity of the isometry group of~$B_\Phi$.
Given a Lie subgroup $H_\Phi$ of $S_\Phi$ acting with cohomogeneity one on $B_\Phi$, the action of the Lie group 
\[
H_\Phi^\Lambda=H_\Phi A_\Phi N_\Phi
\]
on $M$ is of cohomogeneity one on $M$, and is called the \emph{canonical extension} of the $H_\Phi$-action on $B_\Phi$ to $M$. Items (CEI) and (CER) of Theorem~\ref{th:classification} correspond to canonical extensions of two different types of cohomogeneity one $H_\Phi$-actions with a totally geodesic singular orbit on a boundary component $B_\Phi$:
\begin{enumerate}\setcounter{enumi}{2}
	\item[\rm (CEI)] $\Phi$ is a connected subset of simple roots in the Dynkin diagram of $\g{g}$, or equivalently, $B_\Phi$ is an irreducible symmetric space of noncompact type. Up to orbit equivalence, the group $H_\Phi$ can be taken as a maximal proper connected reductive subgroup of $S_\Phi$. As already mentioned, cohomogeneity one actions with a totally geodesic singular orbit on irreducible spaces have been completely classified in~\cite{BT:tohoku}.
	\item[\rm (CER)] $\Phi=\{\alpha,\beta\}$, where $\alpha$ and $\beta$ are orthogonal simple roots with $\dim\g{g}_\alpha=\dim\g{g}_\beta$ and $\dim\g{g}_{2\alpha}=\dim\g{g}_{2\beta}$. Equivalently, $B_\Phi=B_{\{\alpha\}}\times B_{\{\beta\}}=\mathbb{F} \s{H}^n\times \mathbb{F} \s{H}^n$ is a reducible boundary component of rank two whose factors are mutually homothetic hyperbolic spaces of the same dimension and over the same normed division algebra $\mathbb{F}$. By cohomogeneity one \emph{diagonal action} on such a $B_\Phi$ we understand the action of the connected subgroup $H_{\Phi}$ of $G$ with Lie algebra $\g{h}_{\Phi}=\{X+\sigma X:X\in \g{s}_{\{\alpha\}}\}\cong\g{s}_{\{\alpha\}}\cong\g{s}_{\{\beta\}}$, for some Lie algebra isomorphism $\sigma\colon \g{s}_{\{\alpha\}}\to\g{s}_{\{\beta\}}$ between the Lie algebras of the isometry groups of $B_{\{\alpha\}}$ and $B_{\{\beta\}}$. Such an $H_\Phi$-action on $B_\Phi$ has a diagonal totally geodesic singular orbit $H_\Phi\cdot o$ homothetic to $B_{\{\alpha\}}\cong B_{\{\beta\}}$.
\end{enumerate}

In order to describe the nilpotent construction we can and will restrict our attention to subsets $\Phi$ of simple roots of the form $\Phi=\Lambda\setminus\{\alpha\}$, for some simple root $\alpha$. Given any $\lambda\in \Sigma^+\setminus\Sigma^+_{\Phi}$, the coefficient of $\alpha$ in the expression of $\lambda$ as a sum of simple roots is a positive integer $k$. This determines a grading $\g{n}_{\Phi}=\bigoplus_{k\in\mathbb{N}}\g{n}_{\Phi}^k$. Define $L_{\Phi}=Z_G(\g{a}_{\Phi})$ as the centralizer of $\g{a}_{\Phi}$ in $G$, whose Lie algebra is $\g{l}_{\Phi}=\g{g}_0\oplus(\bigoplus_{\lambda\in\Sigma_{\Phi}}\g{g}_\lambda)$, and consider the group $K_{\Phi}=L_{\Phi}\cap K$ and the totally geodesic submanifold $F_{\Phi}=L_{\Phi}\cdot o\cong (A_{\Phi}\cdot o)\times B_{\Phi}$. 

\begin{enumerate}\setcounter{enumi}{4}
	\item[\rm (NC)] Let $\g{v}$ be a subspace of $\g{n}_{\Phi}^1$ with $\dim\g{v}\geq 2$. Let $N_\cdot(\cdot)$ denote a normalizer, and assume that the following two conditions are satisfied:
	\begin{enumerate}[\hspace{7ex}]
		\item [\quad\rm (NC1)] $N_{L_{\Phi}}(\g{n}_{\Phi}\ominus\g{v})$ acts transitively on $F_{\Phi}$, and
		\item[\rm (NC2)] $N_{K_{\Phi}}(\g{v})$ acts transitively on the unit sphere of $\g{v}$. 
	\end{enumerate}
	Then the connected subgroup $H_{\Phi,\g{v}}$ of $G$ with Lie algebra
	\[
	\g{h}_{\Phi,\g{v}}=N_{\g{l}_{\Phi}}(\g{n}_{\Phi}\ominus\g{v})\oplus(\g{n}_{\Phi}\ominus\g{v})
	\]
	acts with cohomogeneity one and a minimal singular orbit $H_{\Phi,\g{v}}\cdot o$ on $M$. In this case, we say that such action has been obtained by \emph{nilpotent construction} from the choice of the simple root $\alpha$ and the subspace $\g{v}$.
\end{enumerate}
As already mentioned, the determination of the possible subspaces $\g{v}$ giving rise to cohomogeneity one actions via nilpotent construction may be a complicated task for many spaces~$M$, due to the difficulty of checking conditions (NC1) and (NC2) simultaneously. It is important to remark that many actions obtained by nilpotent construction may be obtained via canonical extension as well. Indeed, so far, examples of nilpotent constructions that cannot be achieved as canonical extensions have only been found in the hyperbolic spaces of nonconstant curvature, and in the two spaces of $(\s{G}_2)$-type (see~\cite{BT:crelle}, \cite{BD:tg}).
\medskip

As a first application of the structural result in Theorem~\ref{th:classification}, we derive the explicit clas\-sification of cohomogeneity one actions on the symmetric space $\s{SL}_{n+1}(\R)/\s{SO}_{n+1}$, $n\geq 1$, whose rank is $n$. We recall that this family of spaces of noncompact type is universal in the sense that any symmetric space of noncompact type (maybe after rescaling the metric on its irreducible factors) can be isometrically embedded in  $\s{SL}_{n+1}(\R)/\s{SO}_{n+1}$ in an equivariant and totally geodesic manner, for some $n\geq 1$ (see~\cite[\S2.6.5]{Eb:book}). 

\begin{theo}\label{th:sl_n}
	Let $M=G/K=\s{SL}_{n+1}(\R)/\s{SO}_{n+1}$, $n\geq 1$, and let $\Lambda=\{\alpha_1,\dots,\alpha_n\}$ be a set of simple roots for $\g{g}=\g{sl}_{n+1}(\R)$ whose Dynkin diagram is
	
	
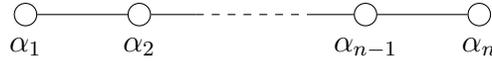
\begin{center}
		\medskip
	\begin{tikzpicture}[node distance=\nodedistance,g/.style={circle,inner sep=\circleinnersep,draw}]
	\node[g] (a1) [label=below:$\alpha_1$] {};
	\node[g] (a2) [right=of a1,label=below:$\alpha_2$] {}
	edge [] (a1);
	\node[] (a3) [right=of a2] {};
	\node[g] (ap) [right=of a3, label=below:$\alpha_{n-1}$] {};
	\node[g] (an) [right=of ap, label=below:$\alpha_n$] {}
	edge [] (ap);
	\draw ($(a3)!.5!(ap)$) -- (ap);
	\draw ($(a3)!.5!(a2)$) -- (a2);
	\draw [dashed] ($(a3)!.5!(a2)$) -- ($(a3)!.5!(ap)$);
	\end{tikzpicture}
	\end{center}
	Any cohomogeneity one action on $M$ is orbit equivalent to one of the following actions:	
	\begin{enumerate}
		\item[\rm (FH)] The action of the connected subgroup of $\s{SL}_{n+1}(\R)$ with Lie algebra 
		$(\g{a}\ominus\ell)\oplus\g{n}$,
		for some one-dimensional linear subspace $\ell$ of $\g{a}$. 
		\item[\rm (FS)] The action of the connected subgroup of $\s{SL}_{n+1}(\R)$ with Lie algebra
		$\g{a}\oplus(\g{n}\ominus\g{g}_{\alpha_j})$,
		for some simple root $\alpha_j\in \Lambda$. 
		\item[\rm (CE)] The canonical extension $H_\Phi^\Lambda$ of the action of the connected subgroup $H_\Phi$ of $G$ on a boundary component $B_\Phi$, for one of the cases listed on the table below.
	\end{enumerate}
			\begin{center}
	\medskip
	\begin{tabular}{lllll}
		\hline
		$\g{h}_\Phi$ & $\Phi$ & $B_\Phi$ & $\codim(H_\Phi^\Lambda\cdot o)$ & \emph{Comments} 
		\\ \hline
		$\g{k}_{\{\alpha_j\}}\cong\g{so}_2$ & $\{\alpha_j\}$ & $\R \s{H}^2$  & $2$ & $1\leq j\leq n$
		\\			
		$\g{sl}_{k-j+1}(\R)\oplus\R$ & $\{\alpha_j,\dots,\alpha_k\}$ & 	$\s{SL}_{k-j+2}(\R)/\s{SO}_{k-j+2}$  & $k-j+1$ & $1\leq j< k \leq n$ 
		\\
		$\g{sp}_{2}(\R)$ & $\{\alpha_j,\alpha_{j+1},\alpha_{j+2}\}$ & 	$\s{SL}_{4}(\R)/\s{SO}_{4}$ & $3$& $1\leq j\leq n-2$ 
		\\
		$\g{s}_{j,k,\sigma}\cong\g{sl}_{2}(\R)$ & $\{\alpha_j,\alpha_k\}$ & 	$\R \s{H}^2\times\R \s{H}^2$ & $2$ & $|k-j|>1$   \\ \hline				
	\end{tabular}	
\end{center}
\end{theo}
		
\medskip
In the table, we use the notation  $\g{s}_{j,k,\sigma}=\{X+\sigma X: X\in\g{s}_{\{\alpha_j\}}\}$ for some isomorphism $\sigma\colon\g{s}_{\{\alpha_j\}}\to \g{s}_{\{\alpha_k\}}$ between the Lie algebras of the isometry groups of $B_{\{\alpha_j\}}$ and $B_{\{\alpha_k\}}$. Without loss of generality, the only singular orbit of the $H_\Phi^\Lambda$-action on $M$ is assumed to pass through the base point $o$. It is important to remark that the group $H_\Phi$ does \emph{not} have to be ``canonically embedded" or ``embedded in the standard way" into the isometry group $S_\Phi$ of $B_\Phi$: its Lie algebra is only of the form $\tau(\g{h}_\Phi^{\mathrm{standard}})$, for some automorphism $\tau$ of $\g{s}_\Phi$, and where $\g{h}_\Phi^{\mathrm{standard}}$ denotes a standard matrix group embedding. This is important since we have a priori no guarantee that the canonical extensions of orbit equivalent actions (for instance, those corresponding to $\g{h}_\Phi^{\mathrm{standard}}$ and $\tau(\g{h}_\Phi^{\mathrm{standard}})$) are orbit equivalent.

\begin{remark}\label{rem:orbit_equivalence}(\emph{Orbit equivalence of the examples.})
	Although the classification results in this article are obtained up to orbit equivalence, the explicit determination of the moduli space of cohomogeneity one actions on a given space entails an added difficulty whose solution lies outside the scope of this article. The reason is that two orbit equivalent cohomogeneity one actions with a totally geodesic singular orbit on a boundary component $B_\Phi$ may (in principle) produce non-orbit equivalent canonical extended actions on $M$. This can happen if the orbit equivalence in $B_\Phi$ is only obtained via an outer isometry of $B_\Phi$ (that is, an isometry not lying in the connected component of the identity of the isometry group of $B_\Phi$), since such outer isometry might not be the restriction of an isometry of $M$. Berndt and Tamaru's classification of cohomogeneity one actions with a totally geodesic singular orbit on irreducible symmetric spaces~\cite{BT:tohoku} is given up to orbit equivalence by a possibly outer isometry. But we do not know if considering the relation of orbit equivalence by an inner isometry (which Solonenko called ``strong orbit equivalence" in~\cite{Solonenko}) would yield more classes. Addressing this problem would require, in particular, understanding the analogous congruence problem for reflective totally geodesic submanifolds, as \cite{BT:tohoku} rests in part on Leung's classification of such submanifolds~\cite{Leung}, where again only congruence by the full group of isometries is considered. In short, this difficulty concerns the actions of type (CEI), as well as of type (CER), where there is a similar problem. The moduli space of actions of foliation types (FH) and (FS) has been determined in~\cite{BT:jdg} in the irreducible case, and in~\cite{DS:isoparametric} and~\cite{Solonenko:reducible} in the reducible setting, whereas the orbit equivalence involving actions obtained by nilpotent construction may in principle require a case-by-case study.
\end{remark}

As a second application of our structural result we reduce the classification problem of cohomogeneity one actions (up to orbit equivalence) on a reducible symmetric space of noncompact type to the classification problem on each one of its irreducible factors. The result basically says that if the action is not of (FH) or (CER) types, then it is a product action. It is interesting to point out that there is no known analog of Theorem~\ref{th:reducible} below in the compact setting, cf.~\cite{K:reducible}.

\begin{theo}\label{th:reducible}
	Let $M$ be a symmetric space of noncompact type with De Rham decomposition $M=M_1\times\dots\times M_s$, where $M_i=G_i/K_i$, $i=1,\dots,s$, and let $G=\prod_{i=1}^s G_i$. 
	Then, a cohomogeneity one action on $M$ is orbit equivalent to one of the following actions:
	\begin{enumerate}[\hspace{8ex}]
		\item [\rm (Prod)] The product action of a subgroup $H_j\times \prod_{\begin{subarray}{c}i=1\\i\neq j\end{subarray}}^s G_i$ of $G$, where $H_j$ is a connected Lie subgroup of $G_j$ that acts with cohomogeneity one on the irreducible factor $M_j$.
		\item [\rm (FH)] The action of the connected subgroup of $G$ with Lie algebra $\g{h}=(\g{a}\ominus\ell)\oplus\g{n}$, for some one-dimensional linear subspace $\ell$ of $\g{a}$. 
		\item [\rm (CER)] The canonical extension of a cohomogeneity one diagonal action on a reducible rank two boundary component of $M$ whose two factors are homothetic.
	\end{enumerate}
\end{theo}

Since actions of types (FH) and (CER) are well understood, Theorem~\ref{th:reducible} easily allows to derive explicit classifications on any product of irreducible spaces $M=M_1\times \dots\times M_s$, whenever we know the classification of cohomogeneity one actions up to orbit equivalence on each irreducible factor $M_i$, $i=1,\dots, s$. This is the case, in particular, of the rank one symmetric spaces of noncompact type. These are precisely the hyperbolic spaces $\mathbb{F} \s{H}^n$, $\mathbb{F}\in\{\R,\C,\H,\mathbb{O}\}$, $n\geq 2$, over the normed division algebras of the reals $\R$, the complex numbers $\C$, the quaternions $\H$, and the octonions $\mathbb{O}$ (in this case, $n=2$). We recall that the real hyperbolic spaces $\R\s{H}^n$ have a root system of type $(\s{A}_1)$, whereas the other rank one symmetric spaces $\mathbb{F} \s{H}^n$, $\mathbb{F}\neq\R$, $n\geq 2$, have a root system of type $(\s{BC}_1)$. Thus, the set of simple roots associated with a product $M=M_1\times\dots\times M_r$, where $M_i=G_i/K_i=\mathbb{F}_i \s{H}^{n_i}$, consists of $r$ mutually orthogonal roots, $\Lambda=\{\alpha_1,\dots,\alpha_r\}$. We will also denote by $\g{k}_i\oplus\g{a}_i\oplus\g{n}_i$ the Iwasawa decomposition of the Lie algebra $\g{g}_i$ of $G_i$ (in particular, $\g{n}_i=\g{g}_{\alpha_i}\oplus\g{g}_{2\alpha_i}$), and put $(\g{k}_i)_0=N_{\g{k}_i}(\g{a}_i)$. In this context, the application of Theorem~\ref{th:reducible} leads to the following classification result.

\begin{theo}\label{th:rank1}
	Let $M=M_1\times\dots\times M_r$ be a Riemannian product of rank one symmetric spaces of noncompact type $M_i=G_i/K_i=\mathbb{F}_{i} \s{H}^{n_i}$, where $\mathbb{F}_{i}\in\{\R,\C,\H,\mathbb{O}\}$, $i=1,\dots,r$, and let $G=\prod_{i=1}^r G_i$. 
	Then, a proper isometric action on $M$ is of cohomogeneity one if and only if it is orbit equivalent to the action of the connected subgroup $H$ of $G$ with one of the following Lie algebras:
		\begin{center}
		\begin{tabular}{lll}
		\hline
			\emph{Type} & $\g{h}$ & \emph{Comments} 
			\\ \hline
			{\rm (FH)} &  $(\g{a}\ominus\ell)\oplus\g{n}$ & $\ell\subset \g{a}$, $\dim\ell=1$. 
			\\[0.5ex] %\hline
			{\rm (FS)} & $\g{a}\oplus(\g{n}\ominus\ell)$ & $\ell\subset\g{g}_{\alpha_j}$, $\dim\ell=1$, $\alpha_j\in \Lambda$. 
			\\[0.5ex] %\hline
			\multirow{2}{*}{\rm (CEI)} &  
			\multirow{2}{*}{$\bigoplus_{\begin{subarray}{c}i=1\\i\neq j\end{subarray}}^r\g{g}_i\oplus\g{h}_j$} & $H_j\subset G_j$ acts on $M_j$ with cohom.~1\\[-0.35ex]&&and a totally geodesic singular orbit.
			\\[0.5ex] %\hline
			\multirow{2}{*}{{\rm (CER)}} & \multirow{2}{*}{$\bigoplus_{\begin{subarray}{c}i=1\\i\neq j,k\end{subarray}}^r\g{g}_i\oplus\g{g}_{j,k,\sigma}$} & $\g{g}_{j,k,\sigma}=\{X+\sigma X:X\in\g{g}_j\}$, $j\neq k$,\\[-0.35ex] && $\sigma\colon\g{g}_j\to\g{g}_k$ Lie algebra isomorphism.
			\\[0.5ex] %\hline
			\multirow{2}{*}{{\rm (NC)}} &
			\multirow{2}{*}{$\bigoplus_{\begin{subarray}{c}i=1\\i\neq j\end{subarray}}^r\g{g}_i\oplus N_{(\g{k}_j)_0}(\g{v})\oplus\g{a}_j\oplus(\g{n}_j\ominus\g{v})$}
			&  $\g{v}\subset \g{g}_{\alpha_j}$ protohomogeneous subspace,\\[-0.35ex]
			&& $\alpha_j\in \Lambda$, $\dim \g{v}\geq 2$.
		   \\ \hline				
		\end{tabular}	
		\smallskip
	\end{center}
\end{theo}
As already mentioned, the cohomogeneity one $H_j$-actions with a totally geodesic singular orbit on a rank one space $M_j$ mentioned in item (CEI) above are well known (up to orbit equivalence), see~\cite[\S6]{BB:crelle}. The protohomogeneous subspaces $\g{v}$ of $\g{g}_{\alpha_j}\cong\mathbb{F}_j^{n_j-1}$ are, by definition, those subspaces such that $N_{(K_j)_0}(\g{v})$ acts transitively on the unit sphere of $\g{v}$, where $(K_j)_0=N_{K_j}(\g{a}_j)$ (equivalently, $\g{v}$ is protohomogeneous if it satisfies condition (NC2) of the nilpotent construction). Protohomogeneous subspaces in the rank one setting have been classified in~\cite[\S7]{BB:crelle} and~\cite[\S4]{BT:tams} for $\mathbb{F}_j\in\{\R,\C,\mathbb{O}\}$, and in~\cite{DDR:crelle} for $\mathbb{F}_j=\H$.

It is important to remark that, unlike the result for $\s{SL}_n(\R)/\s{SO}_n$, in the case considered in Theorem~\ref{th:rank1} one can easily determine when two given actions are orbit equivalent. Indeed, on the one hand, two orbit equivalent actions must be of the same type in Theorem~\ref{th:rank1}, except when $\g{v}\cong\mathbb{F}_j^{l}\subset\mathbb{F}_j^{n_j-1}$, $l\in\{0,\dots,n_j-2\}$ in type (NC) (which yields also an action of type (CEI)).
On the other hand, the moduli space of cohomogeneity one actions on rank one spaces up to orbit equivalence has been completely determined~\cite{BT:tams},~\cite{DDR:crelle} (which immediately gives the moduli space of actions of types (FS), (CEI) and (NC) in Theorem~\ref{th:rank1}, since all these fit into type (Prod) of Theorem~\ref{th:reducible}), as well as for actions of (FH) type~\cite{BT:jdg}, \cite{DS:isoparametric}, \cite{Solonenko:reducible}. Finally, two actions of type (CER) with the same pair $(j,k)$ are orbit equivalent (with independence of the isomorphism $\sigma$), see~Proposition~\ref{prop:two_isomorphisms}. We illustrate how this determination of the moduli space can be carried out by considering the case of the product of two real hyperbolic spaces.

\begin{example} (\emph{Cohomogeneity one actions on $M=\R\s{H}^n\times \R\s{H}^m$.})
	Assume first that $m=n$. Then, the moduli space of cohomogeneity one actions up to orbit equivalence is $(I_n\times\Gamma_1)\sqcup \mathbb{R}\s{P}^1/\Gamma_2\sqcup \{\g{g}_{1,2,\sigma}\}$, where $I_k=\{0,\dots, k-1\}$, and $(\Gamma_1,\Gamma_2)=(\{0\},\mathbb{Z}_2)$ if both $\R\s{H}^n$ factors are isometric, or $(\Gamma_1,\Gamma_2)=(\mathbb{Z}_2,\{\id\})$ otherwise. 
	Given $H_1=\s{SO}_{1,k}^0\times \s{SO}_{n-k}$ (whose action on $\R \s{H}^n$ has a totally geodesic orbit homothetic to $\R \s{H}^k$) and $H_2=\s{SO}_{1,n}^0$, $(k,0)\in I_n\times \Gamma_1$ represents the $H_1\times H_2$-action, and $(k,1)\in I_n\times \Gamma_1$ the $H_2\times H_1$-action. Both actions are orbit equivalent if and only if both $\R\s{H}^n$ factors of $M$ are isometric, which motivates the definition of $\Gamma_1$.
	The quotient $\mathbb{R}\s{P}^1/\Gamma_2$ represents the actions of type (FH), where the space of lines $\ell$ in $\g{a}$ is represented by $\mathbb{R}\s{P}^1$, and $\Gamma_2$ is the group of automorphisms of $\g{a}$ of the form  $\Ad(k)\vert_{\g{a}}$ and inducing a symmetry of the Dynkin diagram of $M$, where $k$ is an isometry of $M$ fixing $o$ and such that $\Ad(k)\g{a}\subset\g{a}$ (see~\cite{DS:isoparametric}). Finally, $\{\g{g}_{1,2,\sigma}\}$ represents the unique diagonal action of type (CER). If both factors have different dimensions $n$ and $m$, the moduli space is $I_n\sqcup I_m\sqcup \mathbb{R}\s{P}^1$.
\end{example}

The tools developed in this paper can be applied to derive explicit classifications on other symmetric spaces of noncompact type. Basically, the only difficulty to do this stems from determining the actions that arise via nilpotent construction. Even in the seemingly simpler case of spaces whose isometry Lie algebra is split real semisimple, this study would entail a long, case-by-case analysis involving various representations of real semisimple Lie algebras. In other cases, the problem seems to get even harder, as illustrated by the solution to the problem for quaternionic hyperbolic spaces~\cite{DDR:crelle}. However, we expect that combining the structural result in Theorem~\ref{th:classification} with an appropriate generalization of the ideas developed by Solonenko in~\cite{Solonenko} (which ultimately rely on~\cite{BD:tg} and~\cite{DDR:crelle}) may lead in the future to the complete solution of the classification problem.

\medskip
This paper is organized as follows. In Section~\ref{sec:parabolic} we introduce the concepts and facts needed to state Berndt and Tamaru's structural result for cohomogeneity one actions. In Section~\ref{sec:diagonal} we discuss diagonal cohomogeneity one actions on reducible symmetric spaces. Section~\ref{sec:structural} is devoted to the proof of our structural result stated in Theorem~\ref{th:classification}. Finally, in Section~\ref{sec:applications} we will prove Theorems~\ref{th:sl_n},~\ref{th:reducible} and~\ref{th:rank1} as an application of Theorem~\ref{th:classification}.

\medskip
The authors would like to thank Alberto Rodr\'iguez-V\'azquez for helpful comments.

\section{Parabolic subgroups and Berndt-Tamaru's result}\label{sec:parabolic}
The aim of this section is to explain the different types of actions considered by Berndt and Tamaru in~\cite{BT:crelle}, as well as their structural result for cohomogeneity one actions (\S\ref{subsec:classes}). For that, we will first introduce in~\S\ref{subsec:parabolic} the basic facts, terminology and notation in relation to the algebraic structure of symmetric spaces of noncompact type, and particularly, the description of parabolic subgroups of real semisimple Lie groups. We will essentially follow the notation in \cite{BT:crelle}; see also \cite[Chapter~13]{BCO:book}, \cite{BDT:jdg}, \cite{BD:tg}, \cite[\S{}I.1]{BJ:book}, \cite[Chapter~2]{Eb:book}, \cite[Chapter~7]{Knapp} and~\cite{Solonenko} for more information.

\subsection{Parabolic subgroups}\label{subsec:parabolic}
Let $M=G/K$ be a connected Riemannian symmetric space of noncompact type. We can assume that $(G,K)$ is a symmetric pair, which in particular implies that $G$ is a Lie group that acts almost effectively on $M$, and $K$ is the isotropy subgroup of $G$ at some point $o\in M$ that we fix from now on. Since $M$ is of noncompact type, the Lie group $G$ is real semisimple, and $K$ is a maximal compact subgroup of $G$. As usual, we will use gothic letters for the Lie algebras. Thus, let $\g{g}=\g{k}\oplus\g{p}$ be a Cartan decomposition of the real semisimple Lie algebra $\g{g}$ of $G$, where the subspace $\g{p}$ is naturally identified with the tangent space $T_o M$. Let $\theta$ be the associated Cartan involution, given by $\theta (X+Y)=X-Y$ for $X\in\g{k}$ and $Y\in\g{p}$, and $\cal{B}$ the Killing form of $\g{g}$. Then $\langle X,Y\rangle=-\cal{B}(X, \theta Y)$ is a positive definite inner product on $\g{g}$ such that $\langle \ad(X) Y, Z\rangle=-\langle Y,\ad(\theta X)Y\rangle$ for every $X$, $Y$, $Z\in\g{g}$. 
From now on, we will consider $\g{g}$ endowed with this inner product. Also, given two subspaces $V\subset W\subset \g{g}$, we will denote by $W\ominus V$ the orthogonal complement of $V$ in $W$ with respect to $\langle \cdot,\cdot\rangle$.

Let $\g{a}$ be a maximal abelian subspace of $\g{p}$, and consider the corresponding restricted root space decomposition
$\g{g}=\g{g}_0\oplus \left( \bigoplus_{\lambda\in\Sigma}\g{g}_\lambda\right)$,
where $\Sigma$ is the set of restricted roots, i.e.\ those nonzero covectors  $\lambda\in \g{a}^*$ such that the subspace 
\[
\g{g}_\lambda=\{X\in\g{g}:[H,X]=\lambda(H)X\text{ for all }H\in\g{a}\}
\]
is nonzero. It turns out that $\g{g}_0=\g{k}_0\oplus\g{a}$, where $\g{k}_0=Z_{\g{k}}(\g{a})$ is the centralizer (and the normalizer) of $\g{a}$ in $\g{k}$. For each root $\lambda\in \Sigma$ we define the root vector $H_\lambda\in\g{a}$ by the relation $\lambda(H)=\langle H_\lambda, H\rangle$ for all $H\in\g{a}$. Moreover, we have  $\theta\g{g}_\lambda=\g{g}_{-\lambda}$ and $[\g{g}_\lambda,\g{g}_\mu]\subset\g{g}_{\lambda+\mu}$, for any  $\lambda,\mu\in\Sigma\cup\{0\}$.

Let $r=\dim\g{a}$ be the rank of $M$. The set $\Sigma$ constitutes a (possibly nonreduced) root system on $\g{a}^*$. Choose a subset $\Sigma^+$ of $\Sigma$ of positive roots, and let $\Lambda=\{\alpha_1,\dots,\alpha_r\}\subset\Sigma^+$ be the corresponding set of simple roots. We define the nilpotent subalgebra $\g{n}=\bigoplus_{\lambda\in\Sigma^+}\g{g}_\lambda$. Then $\g{g}=\g{k}\oplus\g{a}\oplus\g{n}$ is an Iwasawa decomposition of $\g{g}$, and the corresponding decomposition at the Lie group level states that $G$ is diffeomorphic to the Cartesian product $K\times A\times N$, where $A$ and $N$ are the connected subgroups of $G$ with Lie algebras $\g{a}$ and $\g{n}$, respectively. It follows that the solvable part of the Iwasawa decomposition, that is, the connected Lie subgroup $AN$ of $G$ with solvable Lie algebra $\g{a}\oplus\g{n}$, acts simply transitively on $M$. 

A parabolic subalgebra $\g{q}$ of $\g{g}$ is a Lie subalgebra containing $\Ad(g)(\g{k}_0\oplus\g{a}\oplus\g{n})$, for some $g\in G$. Geometrically speaking, and except for $\g{g}$ itself, each parabolic subalgebra of $\g{g}$ is the Lie algebra of the stabilizer $G_x$ of some point at infinity $x$ in the ideal boundary of~$M$. The conjugacy classes of parabolic subalgebras of $\g{g}$ are parametrized by the subsets $\Phi$ of~$\Lambda$. Thus, for any subset $\Phi$ of simple roots, we will denote by $\Sigma_\Phi=\Sigma\cap\spann \Phi$ the root subsystem of $\Sigma$ generated by $\Phi$, and we will put $\Sigma_\Phi^+=\Sigma_\Phi\cap\Sigma^+$. Define the following Lie subalgebras of $\g{g}$:
\[
\g{l}_\Phi=\g{g}_0\oplus\biggl(\bigoplus_{\lambda\in\Sigma_\Phi}\g{g}_\lambda\biggr),\qquad \g{a}_\Phi=\bigcap_{\alpha\in\Phi}\ker\alpha, \qquad \g{n}_\Phi=\bigoplus_{\lambda\in\Sigma^+\setminus\Sigma^+_\Phi}\g{g}_\lambda,
\]
which are reductive, abelian and nilpotent, respectively. 
\begin{remark}
In this article, by a reductive subalgebra of a real semisimple Lie algebra~$\g{g}$ we understand a $\theta$-invariant Lie subalgebra $\g{h}$ of $\g{g}$, for some Cartan involution $\theta$ of $\g{g}$, or equivalently, a subalgebra $\g{h}$ that is canonically embedded with respect to a Cartan decomposition $\g{g}=\g{k}\oplus\g{p}$, namely $\g{h}=(\g{k}\cap\g{h})\oplus(\g{p}\cap\g{h})$. This  implies that $\g{h}$ is a reductive subalgebra in the sense that  $\ad\vert_\g{h}\colon\g{h}\to\g{gl}(\g{g})$ is completely reducible~\cite[\S6]{Bourbaki}, and hence, in particular $\g{h}$ is a reductive Lie algebra. If $G$ is a real semisimple Lie group, we will say that a Lie subgroup $H$ of $G$ is a reductive subgroup if its Lie algebra $\g{h}$ is a reductive subalgebra of $\g{g}$. If $H$ is a reductive subgroup of $G$, the orbit through the base point $o$ that determines the Cartan decomposition is totally geodesic, since $\g{p}\cap\g{h}$ is a Lie triple system.
\end{remark}
Let  $\g{a}^\Phi=\g{a}\ominus\g{a}_\Phi=\bigoplus_{\alpha\in\Phi}\R H_\alpha$ (note that the direct sum in this equation is not necessarily orthogonal). The centralizer and normalizer of $\g{a}_\Phi$ in $\g{g}$ is $\g{l}_\Phi$. Moreover, $[\g{l}_\Phi,\g{n}_\Phi]\subset \g{n}_\Phi$. Then, the Lie algebra $\g{q}_\Phi=\g{l}_\Phi\oplus\g{n}_\Phi$ is the parabolic subalgebra of $\g{g}$ associated with the subset $\Phi$ of~$\Lambda$. The decomposition $\g{q}_\Phi=\g{l}_\Phi\oplus\g{n}_\Phi$ is known as the Chevalley decomposition of $\g{q}_\Phi$.
We also define the reductive subalgebra $\g{m}_\Phi=\g{l}_\Phi\ominus\g{a}_\Phi$ of $\g{g}$, which normalizes $\g{a}_\Phi\oplus\g{n}_\Phi$. 
The decomposition $\g{q}_\Phi=\g{m}_\Phi\oplus\g{a}_\Phi\oplus\g{n}_\Phi$ is called the Langlands decomposition of the parabolic subalgebra $\g{q}_\Phi$. Every parabolic subalgebra of $\g{g}$ is conjugate to some of the subalgebras $\g{q}_\Phi$ for some $\Phi\subset\Lambda$ by means of an element in $K$. We will also consider the subalgebra $\g{k}_\Phi$ of $\g{k}$ given by 
\[
\g{k}_\Phi=\g{q}_\Phi\cap \g{k}=\g{l}_\Phi\cap\g{k}=\g{m}_\Phi\cap\g{k}=\g{k}_0\oplus\biggl(\bigoplus_{\lambda\in\Sigma^+_\Phi}\g{k}_\lambda\biggr),
\]
where  $\g{k}_\lambda=\pi_{\g{k}}(\g{g}_{\lambda})=\g{k}\cap(\g{g}_{-\lambda}\oplus\g{g}_\lambda)$, $\lambda\in\Sigma$, and $\pi_{\g{k}}$ is the orthogonal projection map onto~$\g{k}$. 

The subspace of $\g{p}$ given by
\[
\g{b}_\Phi=\g{m}_\Phi\cap\g{p}=\g{a}^\Phi\oplus\biggl(\bigoplus_{\lambda\in\Sigma^+_\Phi}\g{p}_\lambda\biggr),
\]
where $\g{p}_\lambda=\pi_{\g{p}}(\g{g}_{\lambda})=\g{p}\cap(\g{g}_{-\lambda}\oplus\g{g}_{\lambda})$, is a Lie triple system in $\g{p}$. This means that it corresponds to the tangent space at $o$ of some connected totally geodesic submanifold $B_\Phi$ of $M$. Associated with $\g{b}_\Phi$ one can consider the semisimple Lie algebra $\g{s}_\Phi=[\g{b}_\Phi,\g{b}_\Phi]\oplus \g{b}_\Phi$, where $[\g{b}_\Phi,\g{b}_\Phi]\subset\g{k}_\Phi$. 
Then, $\g{s}_\Phi=[\g{b}_\Phi,\g{b}_\Phi]\oplus \g{b}_\Phi$ is a Cartan decomposition of $\g{s}_\Phi$, and $\g{a}^\Phi$ is a maximal abelian subspace of~$\g{b}_\Phi$. Moreover, the set $\Sigma_\Phi\rvert_{\g{a}^\Phi}=\{\lambda\rvert_{\g{a}^\Phi}:\lambda\in\Sigma_\Phi\}$ is a root system for  $\g{s}_\Phi=[\g{b}_\Phi,\g{b}_\Phi]\oplus \g{b}_\Phi$ with respect to the maximal abelian subspace $\g{a}^\Phi$ of $\g{b}_\Phi$. Since $\lambda\rvert_{\g{a}_\Phi}=0$ for each $\lambda \in\Sigma_\Phi$, the restriction map $\Sigma_\Phi\to \Sigma_\Phi\rvert_{\g{a}^\Phi}$ is bijective. 
Thus we can naturally identify $\Sigma_\Phi$ (resp.\ $\Phi$) with a root system for $\g{s}_\Phi$ (resp.\ with a set of simple roots for $\g{s}_\Phi$) simply by restricting the roots to $\g{a}^\Phi$. We will implicitly do this identification in what follows. For example, if $\lambda\in\Sigma_\Phi$, the root space $(\g{s}_\Phi)_\lambda=(\g{s}_\Phi)_{\lambda\rvert_{\g{a}^\Phi}}$ of $\g{s}_\Phi$ coincides with the root space $\g{g}_\lambda$ of $\g{g}$, and the root space $(\g{s}_\Phi)_0$ of $\g{s}_\Phi$ corresponding to the $0$-weight is $(\g{s}_\Phi)_0=\g{s}_\Phi\cap\g{g}_0=(\g{s}_\Phi\cap\g{k}_0)\oplus\g{a}^\Phi$. In particular, we have the root space decomposition
\[
\g{s}_\Phi=(\g{s}_\Phi)_0\oplus\bigoplus_{\lambda\in \Sigma_\Phi}(\g{s}_\Phi)_\lambda=(\g{s}_\Phi\cap\g{k}_0)\oplus\g{a}^\Phi\oplus\biggl(\bigoplus_{\lambda\in \Sigma_\Phi}\g{g}_\lambda\biggr).
\]

Now we consider some groups associated with the Lie algebras described so far. {Write $A_\Phi$, $N_\Phi$ and $S_\Phi$ for the connected subgroups of $G$ with Lie algebras $\g{a}_\Phi$, $\g{n}_\Phi$ and $\g{s}_\Phi$, respectively. If we define the reductive group $L_\Phi=Z_G(\g{a}_\Phi)$ as the centralizer of $\g{a}_\Phi$ in $G$, then $Q_\Phi=L_\Phi N_\Phi$ is the parabolic subgroup of $G$ associated with the subset $\Phi$ of $\Lambda$. We also define $K_\Phi=L_\Phi\cap K=Z_K(\g{a}_\Phi)$ and $M_\Phi=K_\Phi S_\Phi$. Then $M_\Phi$ is a (possibly disconnected) closed reductive subgroup of $L_\Phi$, $K_\Phi$ is a maximal compact subgroup of $M_\Phi$, and $L_\Phi=M_\Phi \times A_\Phi$.

The orbit $S_\Phi\cdot o$ of the $S_\Phi$-action on $M=G/K$ through $o$ is the totally geodesic submanifold $B_\Phi$ of $M$ with $T_o B_\Phi\cong\g{b}_\Phi$. $B_\Phi$ is itself a symmetric space of noncompact type whose rank agrees with the cardinality $|\Phi|$ of $\Phi$, and is called the boundary component (or boundary symmetric space) of $M$ associated with $\Phi\subset \Lambda$. Moreover, 
\[
B_\Phi=S_\Phi\cdot o=M_\Phi\cdot o\cong M_\Phi/K_\Phi\cong S_\Phi/(S_\Phi\cap K_\Phi).
\]
Since $\g{s}_\Phi$ is $\theta$-invariant and $S_\Phi$ is connected, $(S_\Phi,S_\Phi\cap K)=(S_\Phi,S_\Phi\cap K_\Phi)$ is a symmetric pair, and in particular $\g{s}_\Phi$ is the Lie algebra of the isometry group of $B_\Phi$. 
We also have a diffeomorphism $A_\Phi\times N_\Phi\times M_\Phi\to Q_\Phi$ which induces a diffeomorphism $A_\Phi\times N_\Phi\times B_\Phi\to M$, $(a, n, m\cdot o)\mapsto (anm)\cdot o$, known as the horospherical decomposition of the symmetric space~$M$. Indeed, the action of $A_\Phi N_\Phi$ on $M$ turns out to be free, polar with section $B_\Phi$, and with mutually congruent minimal orbits~\cite{Ta:math_ann}, \cite{DV:imnr}. We recall that a Lie group action is called polar if there is a totally geodesic submanifold (called section) intersecting all orbits, and at every intersection point between an orbit and the section both submanifolds meet perpendicularly. In the case of the $A_\Phi N_\Phi$-action on $M$, the section $B_\Phi$ meets each orbit exactly once.

Let $A^\Phi N^\Phi$ be the connected subgroup of $AN$ with Lie algebra $\g{a}^\Phi\oplus\g{n}^\Phi$, where \[
\g{n}^\Phi=\bigoplus_{\lambda\in\Sigma^+_{\Phi}}\g{g}_\lambda.
\]
Then $B_\Phi=(A^\Phi N^\Phi)\cdot o$, since the solvable part of the Iwasawa decomposition of the real semisimple Lie algebra $\g{s}_\Phi$ is precisely $\g{a}^\Phi\oplus\g{n}^\Phi$, and hence,  $A^\Phi N^\Phi$ acts transitively on $B_\Phi$. Similarly, the connected subgroup $AN^\Phi$ of $AN$ with Lie algebra $\g{a}\oplus\g{n}^\Phi$ acts transitively on the totally geodesic submanifold $F_\Phi=L_\Phi\cdot o\cong(A_\Phi\cdot o) \times B_\Phi$ with Lie triple system~$\g{a}_\Phi\oplus\g{b}_\Phi$.

Later in this paper, we will need to know the relation between the parabolic subalgebras of the semisimple Lie algebra $\g{s}_\Phi$ and the parabolic subalgebras of $\g{g}$. Thus, let  $\Psi\subset \Phi\subset\Lambda$. Then we have the following inclusions of boundary components: $B_\Psi\subset B_\Phi\subset B_\Lambda=M$. By $\g{q}_{\Psi,\Phi}$ we will denote the parabolic subalgebra of $\g{s}_\Phi$ associated with the subset $\Psi$ of the set $\Phi$ of simple roots of $\g{s}_\Phi$. The corresponding Chevalley and Langlands decompositions can then be written as $\g{q}_{\Psi,\Phi}=\g{l}_{\Psi,\Phi}\oplus\g{n}_{\Psi,\Phi}=\g{m}_{\Psi,\Phi}\oplus\g{a}_{\Psi,\Phi}\oplus\g{n}_{\Psi,\Phi}$, where
\begin{align*}
\g{l}_{\Psi,\Phi}&=(\g{s}_\Phi)_0\oplus\Bigl(\bigoplus_{\lambda\in\Sigma_\Psi}\g{g}_\lambda\Bigr),\qquad &\g{n}_{\Psi,\Phi}&=\bigoplus_{\lambda\in\Sigma^+_\Phi\setminus\Sigma^+_\Psi}\g{g}_\lambda=\g{n}^\Phi\cap\g{n}_\Psi,
\\
\g{a}_{\Psi,\Phi}&=\bigcap_{\alpha\in\Psi}\ker \alpha\rvert_{\g{a}^\Phi}
=\g{a}^\Phi\ominus\Bigl(\bigoplus_{\alpha\in \Psi}\R H_\alpha\Bigr)
=\g{a}^\Phi\cap\g{a}_\Psi, \qquad &\g{m}_{\Psi,\Phi}&=\g{l}_{\Psi,\Phi}\ominus\g{a}_{\Psi,\Phi}.
\end{align*}
In particular, we have $\g{q}_{\Psi,\Phi}=\g{q}_\Psi\cap \g{s}_\Phi$. We also define 
\[
\g{k}_{\Psi,\Phi}=\g{k}\cap\g{l}_{\Psi,\Phi}=\g{k}_\Psi\cap\g{s}_\Phi,
\]
and the (possibly disconnected) Lie subgroups $L_{\Psi,\Phi}=Z_{S_\Phi}(\g{a}_{\Psi,\Phi})$, $K_{\Psi,\Phi}=L_{\Psi,\Phi}\cap K$ and $M_{\Psi,\Phi}=K_{\Psi,\Phi} S_\Psi$ of $S_\Phi$, whose respective Lie algebras are $\g{l}_{\Psi,\Phi}$, $\g{k}_{\Psi,\Phi}$, and $\g{m}_{\Psi,\Phi}$.

\subsection{Berndt and Tamaru's result} \label{subsec:classes}
We will now describe the different classes of cohomogeneity one actions that appear in the structure result by Berndt and Tamaru~\cite{BT:crelle}, which is also stated below in Theorem~\ref{th:BT}. We keep on using the notation described above.

\emph{Foliations of horospherical type.} Let $\ell$ be a one-dimensional subspace of $\g{a}$. Then the connected subgroup $H_\ell$ of $G$ with Lie algebra $\g{h}_\ell=(\g{a}\ominus\ell)\oplus\g{n}$ acts on $M$ with cohomogeneity one giving rise to a regular Riemannian foliation whose orbits are congruent to each other. The study of the orbit equivalence of these actions was carried out in~\cite{BT:jdg} for irreducible symmetric spaces $G/K$, and recently in~\cite{DS:isoparametric} and in~\cite{Solonenko:reducible} for the general case. It turns out that two choices $\ell$ and $\ell'$ yield orbit equivalent actions if and only if $\varphi(\ell)=\ell'$ for some linear automorphism $\varphi$ of $\g{a}$ of the form $\varphi=\Ad(k)\vert_\g{a}$ with $k\in N_{I(M)_o}(\g{a})$ and mapping the set $\{H_\lambda:\lambda\in\Sigma^+\}$ onto itself. Here, $I(M)_o$ is the isotropy subgroup at $o$ of the full isometry group $I(M)$ of $M$. We note that an automorphism $\varphi$ as above is precisely an automorphism of $\g{a}$ induced by an isometry of $M$ that is in turn induced by a symmetry of the Dynkin diagram of $\g{g}$. The requirement that $\varphi$ is induced by an isometry of $M$ (i.e.\ $\varphi=\Ad(k)\vert_\g{a}$ with $k\in N_{I(M)_o}(\g{a})$) is superfluous if $M$ is irreducible, but not if $M$ has two homothetic but not isometric factors.

\emph{Foliations of solvable type.} Let $\ell$ be a one-dimensional subspace of a simple root space $\g{g}_{\alpha_j}$, $\alpha_j\in\Lambda$. Then, the connected subgroup $H_j$ of $G$ with Lie algebra $\g{h}_j=\g{a}\oplus(\g{n}\ominus\ell)$ acts on $M$ with cohomogeneity one, and the orbits form a Riemannian foliation with exactly one minimal leaf (the one through $o$). Two lines  $\ell, \ell'$ in the same simple root space $\g{g}_{\alpha_j}$ always yield orbit equivalent actions. More generally, two choices $\ell\subset \g{g}_{\alpha_j}$ and $\ell'\subset \g{g}_{\alpha_k}$ produce orbit equivalent actions if and only if there is an isometry of $M$ induced by a Dynkin diagram symmetry mapping $\alpha_j$ to $\alpha_k$. See \cite{BT:jdg}, \cite{BDT:jdg} and~\cite{Solonenko:reducible} for more information.

\emph{Actions with a totally geodesic singular orbit.} Cohomogeneity one actions with a totally geodesic singular orbit on irreducible symmetric spaces of noncompact type $M$ have been classified up to orbit equivalence (by isometries of the full isometry group $I(M)$) in~\cite{BT:tohoku}. It follows that a totally geodesic submanifold $F$ of an irreducible space $M$ is the singular orbit of a cohomogeneity one action on $M=G/K$ if and only if $F$ is a reflective submanifold whose complementary reflective submanifold has rank one, or $F$ is one of five exceptions (mysteriously related to the group $\s{G}_2$). These actions are induced by maximal proper reductive subgroups of $G$. The converse is not true in general. However, if $L$ is a maximal proper reductive subgroup of $G$, and $H$ is a subgroup of $L$ acting on $M$ with cohomogeneity one, then the actions of $H$ and $L$ have the same orbits, one of which is totally geodesic, being singular if $M$ is irreducible and different from a real hyperbolic~space. 

\emph{Canonical extension.} Consider the Langlands decomposition $Q_\Phi=M_\Phi A_\Phi N_\Phi$ of a maximal proper parabolic subgroup $Q_\Phi$ of $G$ obtained by the choice of some subset $\Phi$ of $\Lambda$. The corresponding boundary component $B_\Phi$ is a noncompact symmetric space of rank $\lvert \Phi\rvert$ embedded in $M$ as a totally geodesic submanifold. Since $\g{s}_\Phi$ is the Lie algebra of the isometry group of $B_\Phi$, it follows that any isometric action (of a connected Lie group) on $B_\Phi$ has the same orbits as the action of some connected Lie subgroup of $S_\Phi$. Let $H_\Phi$ be a Lie subgroup of $S_\Phi$ acting on $B_\Phi$ with cohomogeneity one. 
 Then $H_\Phi^\Lambda=H_{\Phi} A_\Phi N_\Phi$ is a connected Lie subgroup of $Q_\Phi$ acting on $M$ with cohomogeneity one. We say that this action has been obtained by canonical extension of a cohomogeneity one action on the boundary component $B_\Phi$. If two connected closed subgroups $H_{\Phi}$, $H_{\Phi}'$ of $S_\Phi$ act on $B_\Phi$ with cohomogeneity one and their actions are orbit equivalent by an isometry in the connected component of the identity $I^0(B_\Phi)$ of the isometry group of $B_\Phi$ (or equivalently, by an element in $S_\Phi$), then their canonical extensions to $M$ are orbit equivalent by an element of $G$ as well; see~\cite{BT:crelle}. The orbits of the $A_\Phi N_\Phi$-action on $M$ are all minimal, but rarely totally geodesic. In fact, they are totally geodesic if and only if $\Phi$ and $\Lambda\setminus\Phi$ are orthogonal~\cite{Ta:math_ann}. This implies that the canonical extension of an $H_\Phi$-action on $B_\Phi$ will have a minimal orbit if there is a minimal $H_\Phi$-orbit on $B_\Phi$, but will only have a totally geodesic orbit if $\Phi$ and $\Lambda\setminus\Phi$ are orthogonal and there is a totally geodesic $H_\Phi$-orbit on $B_\Phi$, cf.~\cite{DV:imnr}.

\emph{Nilpotent construction.} This construction method was introduced in \cite{BT:crelle} and revisited in~\cite{BD:tg} and~\cite{Solonenko}. Here, it will be enough to consider subsets 
\[
\Phi=\Lambda\setminus\{\alpha_j\}
\]
of $\Lambda$ with cardinality $\lvert \Lambda\rvert-1$.  
Thus, $Q_\Phi=L_\Phi N_\Phi$ is a maximal proper parabolic subgroup of $G$. Recall also that $L_\Phi=M_\Phi A_\Phi$. Consider the vector $H^j\in \g{a}$ such that $\alpha_k(H^j)=\delta_k^j$ is the Kronecker delta of $j$ and $k$. Then $H^j$ induces a gradation $\bigoplus_{\nu\geq 1}\g{n}^\nu_\Phi$ of $\g{n}_\Phi$, where $\g{n}_\Phi^\nu$ is the sum of all root spaces corresponding to positive roots $\lambda\in \Sigma^+$ with $\lambda(H^j)=\nu$. In fact, $\lambda(H^j)=\nu$ if and only if the coefficient of $\alpha_j$ in the expression of $\lambda$ as a linear combination of simple roots is precisely $\nu$. Let $\g{v}$ be a subspace of $\g{n}^1_\Phi$ with $\dim\g{v}\geq 2$. Then $\g{n}_{\Phi,\g{v}}=\g{n}_\Phi\ominus\g{v}$ is a subalgebra of $\g{n}$. Let $N_{\Phi,\g{v}}$ be the corresponding connected subgroup of $N_\Phi$. Denote by $\Theta$ the Cartan involution of $G$ associated with $\theta$. If $\g{v}$ satisfies the following conditions:
\begin{enumerate}[{\rm (NC1)}]
\item  $N^0_{M_\Phi}(\g{n}_{\Phi,\g{v}})=\Theta N^0_{M_\Phi}(\g{v})$ acts transitively on $B_\Phi=M_\Phi \cdot o$, \label{N1}
\item $N^0_{K_\Phi}(\g{n}_{\Phi,\g{v}})=N^0_{K_\Phi}(\g{v})$ acts transitively on the unit sphere of $\g{v}$, \label{N2}
\end{enumerate}
then $H_{\Phi,\g{v}}=N^0_{L_\Phi}(\g{n}_{\Phi,\g{v}})N_{\Phi,\g{v}}=N^0_{M_\Phi}(\g{n}_{\Phi,\g{v}})A_{\Phi}N_{\Phi,\g{v}}$ is a connected subgroup of $Q_\Phi$ that acts on $M$ with cohomogeneity one and singular orbit $H_{\Phi,\g{v}}\cdot o$. (Note that the equalities $N^0_{M_\Phi}(\g{n}_{\Phi,\g{v}})=\Theta N^0_{M_\Phi}(\g{v})$ and $N^0_{K_\Phi}(\g{n}_{\Phi,\g{v}})=N^0_{K_\Phi}(\g{v})$ are satisfied for any $\g{v}$.) We say that the $H_{\Phi,\g{v}}$-action on $M$ has been obtained by nilpotent construction from the choice of $\alpha_j$ and $\g{v}$. Moreover, if $\g{v}$ and $\g{v}'$ are two such subspaces which are conjugate by an element $k\in K_\Phi$, then the cohomogeneity one actions by $H_{\Phi,\g{v}}$ and $H_{\Phi,\g{v}'}$ on $M$ are orbit equivalent via conjugation by $k$. Observe that condition (NC1) in the introduction is slightly different than the one here. The former is quicker to introduce, whereas the latter is more manageable in certain situations. It was shown in~\cite[Proposition~3.2]{BD:tg} that both descriptions are equivalent, due to the fact that 
$N^0_{L_\Phi}(\g{n}_{\Phi,\g{v}})= N^0_{M_\Phi}(\g{n}_{\Phi,\g{v}})A_\Phi$.
We also note that a subspace $\g{v}$ satisfying (NC1) (resp.\ (NC2)) has been called admissible (resp.\  protohomogeneous) in~\cite{Solonenko}.
%\smallskip

The main result of \cite{BT:crelle} guarantees that all cohomogeneity one actions on irreducible symmetric spaces of noncompact type can be obtained by one of the five methods described above. However, note that all these construction methods keep their validity for reducible symmetric spaces.

\begin{theorem}\cite{BT:crelle}\label{th:BT}
	Let $M=G/K$ be a connected irreducible Riemannian symmetric space of noncompact type and rank $r$, and let $H$ be a connected closed subgroup of $G$ acting on $M$ with cohomogeneity one. Then one of the following statements holds:
	\begin{enumerate}[{\rm (1)}]
		\item The orbits form a Riemannian foliation on $M$, and one of the following two cases~holds:
		\begin{enumerate}[{\rm (i)}]
			\item The $H$-action is orbit equivalent to a foliation of horospherical type induced by the action of $H_\ell$ for some one-dimensional subspace $\ell$ of $\g{a}$.
			\item The $H$-action is orbit equivalent to a foliation of solvable type induced by the action of $H_j$ for some $j\in\{1,\dots, r\}$.
		\end{enumerate}
		\item The $H$-action has exactly one singular orbit, and one of the following two cases holds:
		\begin{enumerate}[{\rm (i)}]
			\item $H$ is contained in a maximal proper reductive subgroup $L$ of $G$, the actions of $H$ and $L$ have the same orbits, and the singular orbit is totally geodesic. 
			\item Up to conjugation by an element of $G$, the group $H$ is contained in a maximal proper parabolic subgroup $Q_\Phi$ of $G$, for some $\Phi\subset\Lambda$ with cardinality $\lvert \Phi\rvert=\lvert \Lambda\rvert -1$, and one of the following two subcases holds:
			\begin{enumerate}[{\rm (a)}]
				\item The $H$-action is orbit equivalent to the canonical extension of a cohomogeneity one action with a singular orbit on the maximal proper boundary component $B_\Phi$ of $M$.
				\item The $H$-action is orbit equivalent to the action of a group $H_{\Phi,\g{v}}$ obtained by nilpotent construction, for some subspace $\g{v}\subset\g{n}^1_\Phi$ with $\dim \g{v}\geq 2$.
			\end{enumerate}
		\end{enumerate}
	\end{enumerate}
\end{theorem}
\begin{remark}\label{rem:BDT}
	In case Theorem~\ref{th:BT}~(1), either (1)-(i) or (1)-(ii) holds even if $M$ is not necessarily irreducible. This was proved in~\cite{BDT:jdg}. Moreover, it follows from the proof in~\cite{BDT:jdg} that the orbit equivalence is obtained by an isometry $g\in G$ (i.e., in the connected component of the identity of the isometry group of $M$).
\end{remark}

\begin{remark}\label{rmk:th:BT}
	In case~(2)-(ii) of Theorem~\ref{th:BT}, it is stated that $H\subset gQ_\Phi g^{-1}$, for some $g\in G$. This is not explicitly stated in~\cite[Theorem~1.1]{BT:crelle}, but implicitly understood. This element $g\in G$ is precisely the one that gives the orbit equivalence stated in items (a) and (b) of case~(2)-(ii), as follows from the proof in~\cite{BT:crelle} (specifically, the final paragraph of the proof of~\cite[Theorem~3.2]{BT:crelle}). This actually shows that the orbit equivalence claimed in \cite[Theorem~5.8]{BT:crelle} is in fact an equality of orbit foliations. We also remark that, as stated in~\cite[Theorem~5.8]{BT:crelle}, even without the assumption that $M$ is irreducible, if $H$ is in the situation described in case~(2)-(ii) (i.e., $H$ is a connected subgroup of a maximal proper parabolic subgroup of $G$ and acting with cohomogeneity one on $M$), then one of the cases (a) and (b) above hold.
	This facts will be important later on in our proof of Theorem~\ref{th:classification}.
\end{remark}

\begin{remark}\label{rem:scaling_reducible}
	If $M=M_1\times\dots\times M_s$ is reducible (where each $M_i=G_i/K_i$ is irreducible), any symmetric metric on $M$ is induced by (the restriction to $\g{p}\cong T_oM$ of) a weighted sum $a_1 \cal{B}_{\g{g}_1}+\dots +a_s\cal{B}_{\g{g}_s}$ of the Killing forms of the simple factors of $\g{g}$, where $a_1,\dots,a_s>0$. The proofs of Theorem~\ref{th:BT}~(1) in~\cite{BDT:jdg} and of Theorem~\ref{th:BT}~(2)-(ii) in~\cite[Theorem~5.8]{BT:crelle} assume that the metric on $M$ is induced by the Killing form of $\g{g}$ (i.e., $a_1=\dots=a_s=1$ in the previous expression). However, these results hold for an arbitrary symmetric metric on $M$ (i.e., for any $a_1,\dots,a_s>0$). Indeed, given symmetric metrics $g$, $g'$ on $M=G/K$, we have $G=I^0(M,g)=I^0(M,g')$ (possibly after effectivization), and then, a connected group of isometries of $(M,g)$ acts isometrically with cohomogeneity one on $(M,g)$ if and only if it acts isometrically with cohomogeneity one on $(M,g')$. 	Moreover, the orbit equivalences in Theorem~\ref{th:BT}~(1) and (2)-(ii) are achieved by elements of $G$, as recalled in Remarks~\ref{rem:BDT} and~\ref{rmk:th:BT}, and hence they hold independently of the symmetric metric on $M=G/K$.
\end{remark}

\section{Maximal subgroups and diagonal actions}\label{sec:diagonal}
In this section we show that a group acting with cohomogeneity one on a reducible symmetric space of noncompact type is contained in a maximal proper subgroup that either splits nicely with respect to the decomposition into irreducible factors, or is determined by a diagonal action of a maximal proper reductive subgroup on the product of two rank one irreducible factors.

Let $M=G/K$ be a symmetric space of noncompact type. Let $\g{g}=\g{g}_1\oplus\dots\oplus \g{g}_s$ be the decomposition of the real semisimple Lie algebra $\g{g}$ into simple ideals, and $M=M_1\times \dots \times M_s=G_1/K_1\times \dots \times G_s/K_s$ the corresponding decomposition of $M$ into irreducible symmetric spaces of noncompact type. For each $i\in\{1,\dots,s\}$, we have the Cartan decomposition $\g{g}_i=\g{k}_i\oplus\g{p}_i$.

Let $H$ be a connected closed Lie subgroup of $G$. Let $\g{l}$ be a maximal proper Lie subalgebra of $\g{g}$ containing $\g{h}$ and with corresponding connected Lie subgroup $L$ of $G$. Then, it follows from~\cite[Theorem~15.1, p.~235]{Dynkin} (cf.~\cite[Theorem~2.1]{Kollross:tams}) that either 
\[
\g{l}=\bigoplus_{\begin{subarray}{c}i=1\\i\neq j\end{subarray}}^s \g{g}_i \oplus \g{l}_j
\]
for an index $j\in\{1,\dots, s\}$ and a maximal proper subalgebra $\g{l}_j$ of $\g{g}_j$, or
\[
\g{l}=\bigoplus_{\begin{subarray}{c}i=1\\i\neq j,k\end{subarray}}^s \g{g}_i \oplus \g{g}_{j,k,\sigma},
\]
for two indices $j,k\in\{1,\dots, s\}$, $j\neq k$, an isomorphism $\sigma\colon\g{g}_j\to \g{g}_k$, and where $\g{g}_{j,k,\sigma}=\{X+\sigma X: X\in \g{g}_j\}$. In this case, $\g{g}_{j,k,\sigma}$ and $\g{l}$ are reductive subalgebras of $\g{g}$, and $\g{g}_{j,k,\sigma}$ is a maximal proper reductive subalgebra of $\g{g}_j\oplus\g{g}_k$.

Let us focus on the second case, namely, the maximal proper subalgebra $\g{l}$ has a simple ideal which is diagonal with respect to the decomposition of $\g{g}$ into simple ideals. Let us recall first that there is a natural bijective correspondence between homothety classes of irreducible symmetric spaces of noncompact type and noncompact real simple Lie algebras. Hence, since $\g{g}_j$ and $\g{g}_k$ are isomorphic, the corresponding irreducible symmetric spaces $M_j$ and $M_k$ are homothetic. 
Let $G_{j,k,\sigma}$ be the connected closed subgroup of $G_j\times G_k$ with Lie algebra $\g{g}_{j,k,\sigma}$. Then, according to \cite[Proposition~3.1]{BT:crelle} (it will also follow from Theorem~\ref{th:diag_cohom} below), the action of $G_{j,k,\sigma}$ on $M_j\times M_k$ is not transitive. Hence, since $H\subset L$, if $H$ acts with cohomogeneity one on $M$, the actions of $H$ and $L$  have the same orbits.

It only remains to decide for which real simple Lie algebras $\g{g}_j\cong\g{g}_k$ and corresponding isomorphism $\sigma$ the action of $L$ on $M$ is indeed of cohomogeneity one, and not higher. Equivalently, we have to decide when the action of $G_{j,k,\sigma}$ on $M_j\times M_k$ has cohomogeneity one. The following result answers this question. We recall that an action is said to be hyperpolar if it is polar and its sections are flat.

\begin{theorem}\label{th:diag_cohom}
The action of $G_{j,k,\sigma}$ on $M_j\times M_k$ is hyperpolar and its cohomogeneity coincides with the rank of $M_j$.
\end{theorem}
\begin{proof}
Without loss of generality, we will assume that $\sigma(\g{k}_j)=\g{k}_k$. In other words, the base point $o_k$ we consider in $M_k$ is the one whose isotropy Lie algebra is $\sigma(\g{k}_j)$. Then the Lie algebra of the isotropy group at $(o_j,o_k)$ is $\g{k}_{j,k,\sigma}=\{T+\sigma T: T\in \g{k}_j\}\cong \g{k}_j\cong\g{k}_k$.
Moreover, the orbit of $G_{j,k,\sigma}$ through $(o_j,o_k)\in M_j\times M_k$ is singular and of minimum orbit type, according to the proof of~\cite[Proposition~5.2]{DDK:mathz}.

The cohomogeneity of the action of $G_{j,k,\sigma}$ agrees with the cohomogeneity of the slice representation at $(o_j,o_k)$. We calculate this first.
We have
\[
T_{(o_j,o_k)}(G_{j,k,\sigma}\cdot (o_j,o_k))\cong \pi_{\g{p}_j\oplus\g{p}_k}(\g{g}_{j,k,\sigma})=\{X+\sigma X: X\in\g{p}_j\},
\]
where $\pi_{\g{p}_j\oplus\g{p}_k}=((\id-\theta)/2)\vert_{\g{g}_j\oplus\g{g}_k}$ is the projection map onto $\g{p}_j\oplus\g{p}_k$. For simplicity we will assume that $M_j$ and $M_k$ are isometric, but the proof holds with minor changes if they are only homothetic (cf.~Remark~\ref{rem:factors_hyperpolar}). The normal space to $G_{j,k,\sigma}\cdot (o_j,o_k)$ is 
\[
\nu_{(o_j,o_k)}(G_{j,k,\sigma}\cdot (o_j,o_k))\cong\{X-\sigma X: X\in\g{p}_j\}.
\]
Now, the adjoint action of $\g{k}_{j,k,\sigma}$ on $\nu_{(o_j,o_k)}(G_{j,k,\sigma}\cdot (o_j,o_k))$ is given by
\[
\ad(T+\sigma T)(X-\sigma X)=[T,X]-\sigma[T,X],
\]
for $T+\sigma T\in \g{k}_{j,k,\sigma}$ and $X-\sigma X\in \nu_{(o_j,o_k)}(G_{j,k,\sigma}\cdot (o_j,o_k))$. This representation is clearly equivalent to the adjoint action of $\g{k}_j$ on $\g{p}_j$. Therefore, the slice representation at $(o_j,o_k)$ is equivalent to the isotropy representation of the symmetric space $M_j$, whose cohomogeneity is precisely the rank of $M_j$.

Let $\g{a}_{j,k,\sigma}=\{X-\sigma X:X\in\g{a}_j\}$ and $\Xi=\Exp(\g{a}_{j,k,\sigma})\cdot (o_j,o_k)\subset M_j\times M_k$, where $\g{a}_j$ is a maximal abelian subspace of $\g{p}_j$. As usual we can identify $T_{(o_j,o_k)}\Xi$ with $\g{a}_{j,k,\sigma}$, and this is clearly a section for the slice representation of the  $G_{j,k,\sigma}$-action on $M_j\times M_k$ at $(o_j,o_k)$. It is also clear that $\langle \g{g}_{j,k,\sigma},\g{a}_{j,k,\sigma}\oplus[\g{a}_{j,k,\sigma},\g{a}_{j,k,\sigma}]\rangle=0$, since $\g{a}_{j,k,\sigma}\subset \nu_{(o_j,o_k)}(G_{j,k,\sigma}\cdot (o_j,o_k))$  is abelian.
Then, \cite[Proposition~2.3]{DDK:mathz} guarantees that the $G_{j,k,\sigma}$-action is polar with section $\Xi$. Since $\g{a}_{j,k,\sigma}$ is abelian, then $\Xi$ is flat, which shows that the action is~hyperpolar.
\end{proof}

\begin{remark}
	Being $G_{j,k,\sigma}$ a reductive subgroup of $G_j\times G_k$, its action on $M_j\times M_k$  induces an action on a compact dual symmetric space of  $M_j\times M_k$, see~\cite{K:dual}. Such dual action turns out to be an indecomposable, hyperpolar, Hermann action in the sense of~\cite{K:reducible}.
\end{remark}

\begin{remark}\label{rem:factors_hyperpolar}
The singular orbit of $G_{j,k,\sigma}$ through $(o_j,o_k)\in M_j\times M_k$ is a totally geodesic submanifold of $M_j\times M_k$, since $T_{(o_j,o_k)}(G_{j,k,\sigma}\cdot (o_j,o_k))\cong \{X+\sigma X: X\in\g{p}_j\}$ is a Lie triple system in $\g{p}_j\oplus\g{p}_k$. Intrinsically, this singular orbit is homothetic to $M_j$ and to $M_k$. More specifically, since $\g{g}_j$ and $\g{g}_k$ are isomorphic via $\sigma$, we can assume that their Killing forms are the same; denote both by $\cal{B}$. Then the metrics at $o_j$ and $o_k$ of the irreducible symmetric spaces $M_j$ and $M_k$ can be canonically identified with $\lambda_j \cal{B}\rvert_{\g{p}_j\times \g{p}_j}$ and $\lambda_k \cal{B}\rvert_{\g{p}_k\times \g{p}_k}$, for some positive constants $\lambda_j$, $\lambda_k$. Thus, the metric on  $G_{j,k,\sigma}\cdot (o_j,o_k)$ at $(o_j,o_k)$ is given by $(\lambda_j+\lambda_k)\cal{B}(\pi_{\g{p}_j}(\cdot),\pi_{\g{p}_j}(\cdot))$, where $\pi_{\g{p}_j}\colon\g{p}_j\oplus\g{p}_k\to\g{p}_j$ is the projection onto the first factor. 
\end{remark}

We conclude this section by showing that, up to orbit equivalence in $I(M)$, the role of the automorphism $\sigma$ is irrelevant. More precisely:
\begin{proposition}\label{prop:two_isomorphisms}
	Let $\sigma$, $\tau\colon \g{g}_j\to\g{g}_k$ be two Lie algebra isomorphisms. Then the actions of $G_{j,k,\sigma}$ and of $G_{j,k,\tau}$ on $M_j\times M_k$ are orbit equivalent. Moreover, this orbit equivalence is achieved by means of an element of $G_j\times G_k$ if $\sigma\tau^{-1}$ is an inner automorphism of $\g{g}_k$. 
\end{proposition}
\begin{proof}
	As above, we can assume that $\sigma(\g{k}_j)=\g{k}_k$ and that there exists $g\in G_k$ such that $\Ad(g)\tau(\g{k}_j)=\g{k}_k$, since any two maximal compactly embedded subalgebras of a real semisimple Lie algebra are conjugate by an inner automorphism~\cite[Chapter~VI, \S2]{Helgason}. %\cite[pp.~148-149]{Onishchik}.
	Let $\varphi=\sigma\tau^{-1}\Ad(g^{-1})\in\mathrm{Aut}(\g{g}_k)$. Then $\varphi(\g{k}_k)=\g{k}_k$, and hence $\varphi(\g{p}_k)=\g{p}_k$. Since $\varphi$ is a Lie algebra automorphism, it preserves the Lie bracket, and then also the curvature tensor of $M_k$ at $o_k$, and the Killing form of $\g{g}_k$. Therefore, $\varphi\rvert_{\g{p}_k}\colon\g{p}_k\cong T_{o_k}M_k\to\g{p}_k\cong T_{o_k}M_k$ is a linear isometry that preserves the curvature tensor at $o_k$. Hence, by a well-known result (see~\cite[Corollary~2.3.14]{Wolf}), $\varphi$~is the differential at $o_k$ of an isometry $\psi\in I(M_k)$ that fixes $o_k$. In other words, $\varphi=\Ad(\psi)$, and hence $\sigma=\Ad(\psi g)\tau$, where $\Ad$ is the adjoint representation of the Lie group $I(M_k)$. Then the automorphism $\Ad(\id,\psi g)$ of $\g{g}_j\oplus\g{g}_k$ satisfies $\Ad(\id,\psi g)\g{g}_{j,k,\tau}=\g{g}_{j,k,\sigma}$, and therefore the connected Lie groups $G_{j,k,\sigma}$ and $G_{j,k,\tau}$ are conjugate by the isometry $(\id,\psi g)\in I(M_j\times M_k)$. In particular, their actions on $M_j\times M_k$ are orbit equivalent. Finally, if $\sigma\tau^{-1}$ is inner, then $\varphi$ is also inner, and hence we can assume that $\psi\in G_k$, so $(\id, \psi g)\in G_j\times G_k$.
\end{proof}

\section{New structural result}\label{sec:structural}
The aim of this section is to prove Theorem~\ref{th:classification}, which describes all possible types of cohomogeneity one actions on symmetric spaces of noncompact type. We will first state three lemmas about canonical extensions. The first two lemmas are rather simple, and deal with canonical extensions on products (Lemma~\ref{lemma:product}) and compositions of canonical extensions (Lemma~\ref{lemma:canonical}). The third one (Lemma~\ref{lemma:nilpotent}) is more involved, and describes how canonical extensions of nilpotent constructions look like. Finally, we will prove Theorem~\ref{th:classification}.

If one considers a Riemannian product $M=M_1\times M_2$ of symmetric spaces of noncompact type (where $M_1$ and $M_2$ are not necessarily irreducible), any set of simple roots associated with $M$ is a disjoint union $\Lambda=\Lambda_1\sqcup\Lambda_2$, where each $\Lambda_i$ is a set of simple roots for $M_i$, and where the roots in $\Lambda_1$ are orthogonal to the roots in $\Lambda_2$. The boundary component associated with taking $\Lambda_i$ as a subset of $\Lambda$ turns to be exactly $M_i$, and the canonical extension is in a way well behaved with respect to the Riemannian product.

\begin{lemma}\label{lemma:product}
	Let $M_1=G_1/K_1$ and $M_2=G_2/K_2$ be symmetric spaces of noncompact type, and $M=M_1\times M_2$ their Riemannian product. Let $\Lambda=\Lambda_1\sqcup\Lambda_2$ be a set of simple roots for $\g{g}_1\oplus\g{g}_2$, where $\Lambda_i$ is a set of simple roots for $\g{g}_i$, $i=1,2$. Let $H_{\Lambda_1}$ be a connected Lie subgroup of $G_1$ acting with cohomogeneity one on $M_1$. Then the cohomogeneity one action of $H_{\Lambda_1}\times G_2$ on $M$ has the same orbits as the action of the group $H_{\Lambda_1}^\Lambda$ obtained by canonical extension of $H_{\Lambda_1}$ from the boundary component $B_{\Lambda_1}=M_1$ to $M$.
\end{lemma}
\begin{proof}
	First observe that the roots in $\Lambda_1$ are perpendicular to the roots in $\Lambda_2$. Hence, $\g{a}_{\Lambda_1}=\g{a}^{\Lambda_2}$ and $\g{n}_{\Lambda_1}=\g{n}^{\Lambda_2}$. Thus $\g{h}^{\Lambda}_{\Lambda_1}=\g{h}_{\Lambda_1}\oplus\g{a}_{\Lambda_1}\oplus\g{n}_{\Lambda_1}=\g{h}_{\Lambda_1}\oplus\g{a}^{\Lambda_2}\oplus\g{n}^{\Lambda_2}\subset \g{h}_{\Lambda_1}\oplus\g{g}_2$. Since both $H^{\Lambda}_{\Lambda_1}$ and $H_{\Lambda_1}\times G_2$ act with cohomogeneity one on $M$, we conclude that their actions have the same orbits.
\end{proof}

The next lemma basically states that an iterated canonical extension is a canonical extension itself.

\begin{lemma}\label{lemma:canonical}
Let $\Psi\subset \Phi\subset \Lambda$ be subsets of the set of simple roots $\Lambda$, and let $H_\Psi$ be a subgroup of $S_\Psi$ acting on the boundary component $B_\Psi$ with cohomogeneity one. Denote by $H_\Psi^\Phi$ the canonical extension of $H_\Psi$ from $B_\Psi$ to $B_\Phi$, by $(H_\Psi^\Phi)^\Lambda$ the canonical extension of $H_\Psi^\Phi$ from $B_\Phi$ to $M$, and by $H_\Psi^\Lambda$ the canonical extension of $H_\Psi$ from $B_\Psi$ to $M$. Then $(H_\Psi^\Phi)^\Lambda=H_\Psi^\Lambda$.
\end{lemma}
\begin{proof}
This is a straightforward calculation at the Lie algebra level. First, we have 
\begin{align*}
(\g{h}_\Psi^\Phi)^\Lambda & =\g{h}_\Psi^\Phi\oplus \g{a}_\Phi\oplus\g{n}_\Phi = (\g{h}_\Psi\oplus\g{a}_{\Psi,\Phi}\oplus\g{n}_{\Psi,\Phi}) \oplus \g{a}_\Phi\oplus\g{n}_\Phi.
\end{align*}
But
\[
\g{a}_\Phi\oplus\g{a}_{\Psi,\Phi} = (\g{a}\ominus\g{a}^\Phi)\oplus(\g{a}^\Phi\cap \g{a}_\Psi)=\g{a}_\Psi
\]
and, since $(\g{s}_\Phi)_\lambda=\g{g}_\lambda$ for any $\lambda\in \Sigma_\Phi$,
\[
\g{n}_\Phi\oplus\g{n}_{\Psi,\Phi} = \biggl(\bigoplus_{\lambda\in\Sigma^+\setminus\Sigma^+_\Phi}\g{g}_\lambda\biggr)\oplus\biggl(\bigoplus_{\Sigma^+_\Phi\setminus\Sigma^+_\Psi}\g{g}_\lambda\biggr)=\g{n}_\Psi,
\]
which shows that $(\g{h}_\Psi^\Phi)^\Lambda= \g{h}_\Psi\oplus\g{a}_\Psi\oplus\g{n}_\Psi=\g{h}_\Psi^\Lambda$.
\end{proof}

The following result is more complicated than the previous ones, and roughly states that the canonical extension of a nilpotent construction in a boundary component of $M$ is orbit equivalent to a nilpotent construction on $M$.

\begin{lemma}\label{lemma:nilpotent}
Let $\alpha_j\in\Phi\subset\Lambda$. Let
\[
\g{q}_{\Phi\setminus\{\alpha_j\},\Phi}=\g{l}_{\Phi\setminus\{\alpha_j\},\Phi}\oplus\g{n}_{\Phi\setminus\{\alpha_j\},\Phi}
\]
be the Chevalley decomposition of the parabolic subalgebra of $\g{s}_\Phi$ associated with the subset $\Phi\setminus\{\alpha_j\}$ of $\Phi$. Let \[
H_\Phi=H_{\Phi\setminus\{\alpha_j\},\Phi,\g{v}}=N^0_{L_{\Phi\setminus\{\alpha_j\},\Phi}}(\g{n}_{\Phi\setminus\{\alpha_j\},\Phi,\g{v}})N_{\Phi\setminus\{\alpha_j\},\Phi,\g{v}}
\]
be a subgroup of $S_\Phi$ obtained by nilpotent construction acting on $B_\Phi$ with cohomogeneity one, where $\g{v}$ is a subspace of $\g{n}_{\Phi\setminus\{\alpha_j\},\Phi}^1$ and $\g{n}_{\Phi\setminus\{\alpha_j\},\Phi,\g{v}}=\g{n}_{\Phi\setminus\{\alpha_j\},\Phi}\ominus\g{v}$.

Then, $\g{v}\subset\g{n}_{\Lambda\setminus\{\alpha_j\}}^1$, and $H_{\Lambda\setminus\{\alpha_j\},\g{v}}=N^0_{L_{\Lambda\setminus\{\alpha_j\}}}(\g{n}_{\Lambda\setminus\{\alpha_j\},\g{v}})N_{\Lambda\setminus\{\alpha_j\},\g{v}}$ is a subgroup of $G$ obtained by nilpotent construction and acting with cohomogeneity one on $M$. Moreover, the action of $H_{\Lambda\setminus\{\alpha_j\},\g{v}}$ on $M$ has the same orbits as the $H^\Lambda_\Phi$-action obtained by canonical extension of the $H_\Phi$-action from $B_\Phi$ to $M$.
\end{lemma}
\begin{proof}
First of all, we check that $\g{v}$ is indeed a subspace of $\g{n}^1_{\Lambda\setminus\{\alpha_j\}}$. This is easy, since
\[
\g{v}\subset\g{n}_{\Phi\setminus\{\alpha_j\},\Phi}^1=\bigoplus_{\begin{subarray}{c}\lambda\in\Sigma^+_\Phi\\\lambda(H^j)=1\end{subarray}}\g{g}_\lambda\subset \bigoplus_{\begin{subarray}{c}\lambda\in\Sigma^+\\\lambda(H^j)=1\end{subarray}}\g{g}_\lambda=\g{n}^1_{\Lambda\setminus\{\alpha_j\}}.
\]

We will now check that this choice of $\g{v}$ as a subspace of $\g{n}^1_{\Lambda\setminus\{\alpha_j\}}$ gives rise to a nilpotent construction, that is, $\g{v}$ satisfies the conditions (NC1) and (NC2) in~\S\ref{subsec:classes}. 

We have the inclusion
\[
\g{l}_{\Phi\setminus\{\alpha_j\},\Phi}=(\g{s}_\Phi)_0\oplus\biggl(\bigoplus_{\lambda\in\Sigma_{\Phi\setminus\{\alpha_j\}}}\g{g}_\lambda\biggr) 
\subset \g{g}_0\oplus\biggl(\bigoplus_{\lambda\in\Sigma_{\Lambda\setminus\{\alpha_j\}}}\g{g}_\lambda\biggr)=\g{l}_{\Lambda\setminus\{\alpha_j\}},
\]
and hence $\g{k}_{\Phi\setminus\{\alpha_j\},\Phi}=\g{l}_{\Phi\setminus\{\alpha_j\},\Phi}\cap\g{k}\subset\g{l}_{\Lambda\setminus\{\alpha_j\}}\cap\g{k}=\g{k}_{\Lambda\setminus\{\alpha_j\}}$.  Therefore, 
\begin{equation}\label{eq:inclusion_normalizers}
N_{\g{l}_{\Phi\setminus\{\alpha_j\},\Phi}}(\g{n}_{\Phi\setminus\{\alpha_j\},\Phi,\g{v}}) = \theta N_{\g{l}_{\Phi\setminus\{\alpha_j\},\Phi}}(\g{v}) \subset \theta N_{\g{l}_{\Lambda\setminus\{\alpha_j\}}}(\g{v})= N_{\g{l}_{\Lambda\setminus\{\alpha_j\}}}(\g{n}_{\Lambda\setminus\{\alpha_j\},\g{v}}),
\end{equation}
and $N_{\g{k}_{\Phi\setminus\{\alpha_j\},\Phi}}(\g{v})\subset N_{\g{k}_{\Lambda\setminus\{\alpha_j\}}}(\g{v})$.

By hypothesis, $\g{v}\subset\g{n}_{\Phi\setminus\{\alpha_j\},\Phi}^1$ satisfies the two conditions for the nilpotent construction on $B_\Phi$. In particular, $N^0_{K_{\Phi\setminus\{\alpha_j\},\Phi}}(\g{v})$ acts transitively on the unit sphere of $\g{v}$. Since $N^0_{K_{\Phi\setminus\{\alpha_j\},\Phi}}(\g{v})\subset N^0_{K_{\Lambda\setminus\{\alpha_j\}}}(\g{v})$, so does $N^0_{K_{\Lambda\setminus\{\alpha_j\}}}(\g{v})$. Thus $\g{v}$ satisfies condition~(NC2) for the nilpotent construction on $M$.

In order to check condition (NC1), first note that $N^0_{M_{\Lambda\setminus\{\alpha_j\}}}(\g{n}_{\Lambda\setminus\{\alpha_j\},\g{v}})$ leaves $B_{\Lambda\setminus\{\alpha_j\}}=M_{\Lambda\setminus\{\alpha_j\}}\cdot o$ invariant. Therefore, it will be enough to verify the inclusion
\begin{equation}\label{eq:inclusion} \g{b}_{\Lambda\setminus\{\alpha_j\}}\subset T_o \big(N^0_{M_{\Lambda\setminus\{\alpha_j\}}}(\g{n}_{\Lambda\setminus\{\alpha_j\},\g{v}})\cdot o\big)
\end{equation}
to see that $N^0_{M_{\Lambda\setminus\{\alpha_j\}}}(\g{n}_{\Lambda\setminus\{\alpha_j\},\g{v}})$ acts transitively on $B_{\Lambda\setminus\{\alpha_j\}}$. For this, decompose $\g{b}_{\Lambda\setminus\{\alpha_j\}}$ as
\[
\g{b}_{\Lambda\setminus\{\alpha_j\}}=
\g{a}^{\Lambda\setminus\{\alpha_j\}}\oplus\biggl(\bigoplus_{\lambda\in\Sigma^+_{\Lambda\setminus\{\alpha_j\}}}\g{p}_{\lambda}\biggr)=
\g{a}_{\Phi\setminus\{\alpha_j\},\Lambda\setminus\{\alpha_j\}}\oplus\g{b}_{\Phi\setminus\{\alpha_j\}} \oplus T_o (N_{\Phi\setminus\{\alpha_j\},\Lambda\setminus\{\alpha_j\}}\cdot o),
\]
with $\g{a}_{\Phi\setminus\{\alpha_j\},\Lambda\setminus\{\alpha_j\}}=\g{a}^{\Lambda\setminus\{\alpha_j\}}\cap\g{a}_{\Phi\setminus\{\alpha_j\}}=\g{a}^{\Lambda\setminus\{\alpha_j\}}\ominus\g{a}^{\Phi\setminus\{\alpha_j\}}$, $\g{b}_{\Phi\setminus\{\alpha_j\}}=\g{a}^{\Phi\setminus\{\alpha_j\}}\oplus\bigl(\bigoplus_{\lambda\in\Sigma^+_{\Phi\setminus\{\alpha_j\}}}\g{p}_\lambda\bigr)$ and $T_o (N_{\Phi\setminus\{\alpha_j\},\Lambda\setminus\{\alpha_j\}}\cdot o)\cong\bigoplus_{\lambda\in\Sigma^+_{\Lambda\setminus\{\alpha_j\}}\setminus\Sigma^+_{\Phi\setminus\{\alpha_j\}}}\g{p}_\lambda$. 
We will prove~\eqref{eq:inclusion} by showing that each one of the three addends in the right-hand term of the previous relation for $\g{b}_{\Lambda\setminus\{\alpha_j\}}$ is contained in $T_o \big(N^0_{M_{\Lambda\setminus\{\alpha_j\}}}(\g{n}_{\Lambda\setminus\{\alpha_j\},\g{v}})\cdot o\big)$.

By assumption, $N^0_{M_{\Phi\setminus\{\alpha_j\},\Phi}}(\g{n}_{\Phi\setminus\{\alpha_j\},\Phi,\g{v}})$ acts transitively on $B_{\Phi\setminus\{\alpha_j\}}$, so
\[
\g{b}_{\Phi\setminus\{\alpha_j\}}= T_o\bigl(N^0_{M_{\Phi\setminus\{\alpha_j\},\Phi}}(\g{n}_{\Phi\setminus\{\alpha_j\},\Phi,\g{v}})\cdot o\bigr)\subset T_o \big(N^0_{M_{\Lambda\setminus\{\alpha_j\}}}(\g{n}_{\Lambda\setminus\{\alpha_j\},\g{v}})\cdot o\big),
\]
where in the inclusion we have used \eqref{eq:inclusion_normalizers} and $\g{m}_{\Phi\setminus\{\alpha_j\},\Phi}\subset\g{m}_{\Lambda\setminus\{\alpha_j\}}$. 

Let $H\in\g{a}_{\Phi\setminus\{\alpha_j\}}$ and $X=\sum_{\lambda\in\Sigma^+_\Phi,\,\lambda(H^j)=1}X_\lambda\in\g{v}\subset\g{n}_{\Phi\setminus\{\alpha_j\},\Phi}^1$, with $X_\lambda\in\g{g}_\lambda$. Given $\lambda\in\Sigma^+_\Phi$ such that $\lambda(H^j)=1$, we can write $\lambda=\alpha_j+\sum_{\alpha\in\Phi\setminus\{\alpha_j\}}n_\alpha\alpha$, for some $n_\alpha\in \mathbb{Z}_{\geq 0}$. Then $\lambda(H)=\alpha_j(H)$, since $H\in\g{a}_{\Phi\setminus\{\alpha_j\}}$. Thus
\[
[H,X]=\sum_{\begin{subarray}{c}\lambda\in\Sigma^+_\Phi\\\lambda(H^j)=1\end{subarray}} \lambda(H)X_\lambda
=\sum_{\begin{subarray}{c}\lambda\in\Sigma^+_\Phi\\\lambda(H^j)=1\end{subarray}} \alpha_j(H)X_\lambda
=\alpha_j(H) X,
\]
which means that $\g{a}_{\Phi\setminus\{\alpha_j\}}$ normalizes $\g{v}$. Since $\g{a}_{\Phi\setminus\{\alpha_j\}}\subset\g{a}\subset \g{l}_{\Lambda\setminus\{\alpha_j\}}$, this implies
\begin{equation}\label{eq:inclusion_aPhi}
	\g{a}_{\Phi\setminus\{\alpha_j\}}=\theta \g{a}_{\Phi\setminus\{\alpha_j\}}\subset  \theta N_{\g{l}_{\Lambda\setminus\{\alpha_j\}}}(\g{v})
	= N_{\g{l}_{\Lambda\setminus\{\alpha_j\}}}(\g{n}_{\Lambda\setminus\{\alpha_j\},\g{v}}).
\end{equation}	
Intersecting with $\g{a}^{\Lambda\setminus\{\alpha_j\}}\subset \g{m}_{\Lambda\setminus\{\alpha_j\}}$, we get $\g{a}^{\Lambda\setminus\{\alpha_j\}}\cap\g{a}_{\Phi\setminus\{\alpha_j\}}\subset N_{\g{m}_{\Lambda\setminus\{\alpha_j\}}}(\g{n}_{\Lambda\setminus\{\alpha_j\},\g{v}})$. Therefore we deduce $\g{a}_{\Phi\setminus\{\alpha_j\},\Lambda\setminus\{\alpha_j\}}=\g{a}^{\Lambda\setminus\{\alpha_j\}}\cap\g{a}_{\Phi\setminus\{\alpha_j\}}\subset T_o \big(N^0_{M_{\Lambda\setminus\{\alpha_j\}}}(\g{n}_{\Lambda\setminus\{\alpha_j\},\g{v}})\cdot o\big)$. 

We now check that $T_o \big(N_{\Phi\setminus\{\alpha_j\},\Lambda\setminus\{\alpha_j\}}\cdot o\big)\subset T_o\big(N^0_{M_{\Lambda\setminus\{\alpha_j\}}}(\g{n}_{\Lambda\setminus\{\alpha_j\},\g{v}})\cdot o\big)$. We will first prove
\begin{equation}\label{eq:arg_Carlos}
	\g{n}_\Phi\cap\g{m}_{\Lambda\setminus\{\alpha_j\}}\subset N_{\g{m}_{\Lambda\setminus\{\alpha_j\}}}(\g{n}_{\Lambda\setminus\{\alpha_j\},\g{v}}).
\end{equation}
Observe that 
$\g{v}\subset\g{n}^1_{\Phi\setminus\{\alpha_j\},\Phi}\subset\g{m}_\Phi\cap\g{n}_{\Lambda\setminus\{\alpha_j\}}$. Then,
\begin{align*}
[\theta(\g{n}_\Phi\cap\g{m}_{\Lambda\setminus\{\alpha_j\}}),\g{v}]
&\subset [\theta\g{n}_\Phi,\g{v}]\cap [\g{m}_{\Lambda\setminus\{\alpha_j\}},\g{v}]
\subset [\theta\g{n}_\Phi,\g{m}_\Phi]\cap [\g{m}_{\Lambda\setminus\{\alpha_j\}},\g{n}_{\Lambda\setminus\{\alpha_j\}}]
\\
&\subset\theta[\g{n}_\Phi,\g{m}_\Phi]\cap \g{n}_{\Lambda\setminus\{\alpha_j\}}
=\theta\g{n}_\Phi \cap \g{n}_{\Lambda\setminus\{\alpha_j\}}\subset \theta\g{n}\cap\g{n}=0.
\end{align*}
Thus, $\theta(\g{n}_\Phi\cap\g{m}_{\Lambda\setminus\{\alpha_j\}})\subset N_{\g{m}_{\Lambda\setminus\{\alpha_j\}}}(\g{v})=\theta N_{\g{m}_{\Lambda\setminus\{\alpha_j\}}}(\g{n}_{\Lambda\setminus\{\alpha_j\},\g{v}})$, from where~\eqref{eq:arg_Carlos} follows. 
But then $\g{n}_{\Phi\setminus\{\alpha_j\},\Lambda\setminus\{\alpha_j\}}\subset\g{n}_\Phi\cap\g{m}_{\Lambda\setminus\{\alpha_j\}}\subset N_{\g{m}_{\Lambda\setminus\{\alpha_j\}}}(\g{n}_{\Lambda\setminus\{\alpha_j\},\g{v}})$, and hence $T_o \big(N_{\Phi\setminus\{\alpha_j\},\Lambda\setminus\{\alpha_j\}}\cdot o\big)\subset T_o\big(N^0_{M_{\Lambda\setminus\{\alpha_j\}}}(\g{n}_{\Lambda\setminus\{\alpha_j\},\g{v}})\cdot o\big)$. This concludes the proof of~\eqref{eq:inclusion}. Therefore,   $N^0_{M_{\Lambda\setminus\{\alpha_j\}}}(\g{n}_{\Lambda\setminus\{\alpha_j\},\g{v}})$ acts transitively on $B_{\Lambda\setminus\{\alpha_j\}}$, and thus, $\g{v}$ satisfies the condition (NC1) for the nilpotent construction on $M$. Since, as shown above, (NC2) also holds, we get that  $H_{\Lambda\setminus\{\alpha_j\},\g{v}}$ acts on $M$ with cohomogeneity one.

In order to conclude the proof of the lemma, we just have to see that the actions of $H_{\Lambda\setminus\{\alpha_j\},\g{v}}$ and $H^\Lambda_\Phi$ have the same orbits. For this, we will show that $\g{h}^\Lambda_\Phi\subset \g{h}_{\Lambda\setminus\{\alpha_j\},\g{v}}$. First observe that 
\[
\g{h}^\Lambda_\Phi=N_{\g{l}_{\Phi\setminus\{\alpha_j\},\Phi}}(\g{n}_{\Phi\setminus\{\alpha_j\},\Phi,\g{v}})\oplus\g{n}_{\Phi\setminus\{\alpha_j\},\Phi,\g{v}}\oplus\g{a}_{\Phi}\oplus\g{n}_{\Phi},
\]
and recall
\[
\g{h}_{\Lambda\setminus\{\alpha_j\},\g{v}}=N_{\g{l}_{\Lambda\setminus\{\alpha_j\}}}(\g{n}_{\Lambda\setminus\{\alpha_j\},\g{v}})\oplus\g{n}_{\Lambda\setminus\{\alpha_j\},\g{v}}.
\]
We have seen in~\eqref{eq:inclusion_normalizers} that $N_{\g{l}_{\Phi\setminus\{\alpha_j\},\Phi}}(\g{n}_{\Phi\setminus\{\alpha_j\},\Phi,\g{v}}) \subset N_{\g{l}_{\Lambda\setminus\{\alpha_j\}}}(\g{n}_{\Lambda\setminus\{\alpha_j\},\g{v}})$. 
Also, $\g{n}_{\Phi\setminus\{\alpha_j\},\Phi}\subset \g{n}_{\Lambda\setminus\{\alpha_j\}}$, 
and hence $\g{n}_{\Phi\setminus\{\alpha_j\},\Phi,\g{v}}\subset \g{n}_{\Lambda\setminus\{\alpha_j\},\g{v}}$. By~\eqref{eq:inclusion_aPhi} we have $\g{a}_\Phi\subset \g{a}_{\Phi\setminus\{\alpha_j\}}\subset N_{\g{l}_{\Lambda\setminus\{\alpha_j\}}}(\g{n}_{\Lambda\setminus\{\alpha_j\},\g{v}})$.
Finally, we show that $\g{n}_\Phi=\bigoplus_{\lambda\in\Sigma^+\setminus\Sigma^+_\Phi}\g{g}_\lambda$ is contained in $\g{h}_{\Lambda\setminus\{\alpha_j\},\g{v}}$. Let $\lambda\in\Sigma^+\setminus\Sigma^+_{\Phi}$. Assume first $\lambda\notin\Sigma^+_{\Lambda\setminus\{\alpha_j\}}$. Then $\g{g}_\lambda\subset \g{n}_{\Lambda\setminus\{\alpha_j\},\g{v}}$, since $\g{g}_\lambda\perp\g{n}_{\Phi\setminus\{\alpha_j\},\Phi}\supset\g{v}$ as $\lambda\notin\Sigma^+_\Phi$. Now suppose $\lambda\in\Sigma^+_{\Lambda\setminus\{\alpha_j\}}$. Then, by~\eqref{eq:arg_Carlos}, $\g{g}_\lambda\subset\g{n}_\Phi\cap \g{m}_{\Lambda\setminus\{\alpha_j\}}\subset N_{\g{m}_{\Lambda\setminus\{\alpha_j\}}}(\g{n}_{\Lambda\setminus\{\alpha_j\},\g{v}})\subset N_{\g{l}_{\Lambda\setminus\{\alpha_j\}}}(\g{n}_{\Lambda\setminus\{\alpha_j\},\g{v}})$. 

Altogether we have $\g{h}^\Lambda_\Phi\subset \g{h}_{\Lambda\setminus\{\alpha_j\},\g{v}}$, and by connectedness,  $H^{\Lambda}_{\Phi}\subset H_{\Lambda\setminus\{\alpha_j\},\g{v}}$. Since both groups act with cohomogeneity one on $M$, they must have the same orbits.
\end{proof}

We can now prove the first main result of this paper.

\begin{proof}[Proof of Theorem~\ref{th:classification}]
By construction, an action of any of the five types stated in Theorem~\ref{th:classification} is of cohomogeneity one. 

In order to prove the converse, let $H$ be a connected closed subgroup of $G$ acting on $M=G/K$ with cohomogeneity one.
If the action of $H$ has no singular orbits, then the results in \cite{BT:jdg} and \cite{BDT:jdg} (see~Remarks~\ref{rem:BDT} and~\ref{rem:scaling_reducible}) guarantee that the $H$-action is orbit equivalent to one of the actions of foliation type, namely (FH) or (FS).

Hence, we will assume from now on that the $H$-action on $M$ has a singular orbit. Let $\g{g}=\g{g}_1\oplus\dots\oplus \g{g}_s$ be the decomposition of the real semisimple Lie algebra $\g{g}$ into simple ideals, and $M=M_1\times \dots \times M_s=G_1/K_1\times \dots \times G_s/K_s$ the corresponding decomposition of $M$ into irreducible symmetric spaces of noncompact type. Let $\g{q}$ be a maximal proper Lie subalgebra of $\g{g}$ containing the Lie algebra $\g{h}$ of $H$, and let $Q$ be the connected subgroup of $G$ with Lie algebra $\g{q}$. According to the exposition in Section~\ref{sec:diagonal} we must have  $\g{q}=\bigoplus_{\begin{subarray}{c}i=1\\i\neq l\end{subarray}}^s \g{g}_i \oplus \g{q}_l$ for an index $l\in\{1,\dots, s\}$ and a maximal proper subalgebra $\g{q}_l$ of $\g{g}_l$, or $\g{q}=\bigoplus_{\begin{subarray}{c}i=1\\i\neq j,k\end{subarray}}^s \g{g}_i \oplus \g{g}_{j,k,\sigma}$, for two indices $j,k\in\{1,\dots, s\}$, $j\neq k$, and an isomorphism $\sigma\colon\g{g}_j\to \g{g}_k$, where  $\g{g}_{j,k,\sigma}=\{X+\sigma X:X\in\g{g}_j\}$ is a maximal proper reductive subalgebra of $\g{g}_j\oplus\g{g}_k$.  

In the second case, in view of Theorem~\ref{th:diag_cohom}, the cohomogeneity of the $Q$-action on $M$ agrees with the rank of $M_j$, and so it must be equal to one, since $H\subset Q$ acts on $M$ with cohomogeneity one by assumption. Thus, the actions of $H$ and $Q$ have the same orbits. Hence, by Lemma~\ref{lemma:product}, the $H$-action has the same orbits as a canonical extension of a cohomogeneity one diagonal action on the boundary component $B_\Phi=M_j\times M_k$ with $\Phi=\{\beta_j,\beta_k\}\subset\Lambda$,  $B_{\{\beta_j\}}=M_j$, $B_{\{\beta_k\}}=M_k$, $\g{g}_j=\g{s}_{\{\beta_j\}}$, and $\g{g}_k=\g{s}_{\{\beta_k\}}$. Thus, the $H$-action is orbit equivalent to an action of (CER) type.

We consider the first case from now on, i.e.\ $\g{h}\subset\g{q}=\bigoplus_{\begin{subarray}{c}i=1\\i\neq l\end{subarray}}^s \g{g}_i \oplus \g{q}_l$ for an index $l\in\{1,\dots, s\}$ and a maximal proper subalgebra $\g{q}_l$ of $\g{g}_l$. Then, by~\cite[Theorem~3.2]{BT:crelle} (which ultimately relies on the work of Mostow~\cite{Mostow}), $\g{q}_l$ is either a maximal proper reductive or a maximal proper parabolic subalgebra of $\g{g}_l$. If $\g{q}_l$ is a  reductive subalgebra of $\g{g}_l$, then $\g{q}$ is a reductive subalgebra of $\g{g}$, and the $H$-action and the $Q$-action have the same orbits by \cite[Theorem~3.2]{BT:crelle}. By the same result and the assumption that $H$ has a singular orbit, we have that such singular orbit is totally geodesic. Using Lemma~\ref{lemma:product}, we see that the actions of $H$ and $Q$ have the same orbits as an action obtained by the canonical extension of a cohomogeneity one $Q_l$-action with a totally geodesic singular orbit on the irreducible boundary component $B_\Phi=M_l$, where $\Phi$ is the subset of $\Lambda$ consisting of all simple roots of $\g{g}_l$. This corresponds to an action of type (CEI). 

Henceforth, we assume that $\g{q}_l$ is a maximal proper parabolic subalgebra of $\g{g}_l$, and hence, $\g{q}$ is a maximal proper parabolic subalgebra of $\g{g}$. Then, there is an element $g\in G_l\subset G$ such that $\Ad(g)\g{q}$ is a standard maximal proper parabolic subalgebra $\g{q}_{\Lambda\setminus\{\alpha_j\}}$, for some simple root $\alpha_j\in\Lambda$ corresponding to $\g{g}_l\subset \g{g}$. Therefore, we know from \cite[Theorem~5.8]{BT:crelle} (see Remarks~\ref{rmk:th:BT} and~\ref{rem:scaling_reducible}) that the $H$-action on $M$ is orbit equivalent (via $g\in G_l$) to a cohomogeneity one action on $M$ obtained by canonical extension of a cohomogeneity one action on the boundary component $B_{\Lambda\setminus\{\alpha_j\}}$, or to a cohomogeneity one action on $M$ of a group $H_{\Lambda\setminus\{\alpha_j\},\g{v}}$ obtained by nilpotent construction, for some subspace $\g{v}\subset \g{n}^1_{\Lambda\setminus\{\alpha_j\}}$. This second case corresponds to an action of type (NC) in the statement of Theorem~\ref{th:classification}.

Hence, we assume that the $H$-action is orbit equivalent to the canonical extension $H_{\Lambda\setminus\{\alpha_j\}}^\Lambda$ of certain connected closed subgroup $H_{\Lambda\setminus\{\alpha_j\}}\subset S_{\Lambda\setminus\{\alpha_j\}}$ acting on $B_{\Lambda\setminus\{\alpha_j\}}$ with cohomogeneity one. We can and will also assume that the $H_{\Lambda\setminus\{\alpha_j\}}$-action on $B_{\Lambda\setminus\{\alpha_j\}}$ has a singular orbit, since otherwise its canonical extension (and hence, the $H$-action) would yield a homogeneous regular foliation on $M$, contradicting the assumption that the $H$-action has a singular orbit. 

Let $j_1=j$. Now we apply all the procedure described so far (for actions with singular orbits) with $B_{\Lambda\setminus\{\alpha_{j_1}\}}$ instead of $M$ and with $H_{\Lambda\setminus\{\alpha_{j_1}\}}\subset S_{\Lambda\setminus\{\alpha_{j_1}\}}$ instead of $H\subset G$. In the case that the $H_{\Lambda\setminus\{\alpha_{j_1}\}}$-action on $B_{\Lambda\setminus\{\alpha_{j_1}\}}$ is orbit equivalent to a canonical extension of a group $H_{\Lambda\setminus\{\alpha_{j_1},\alpha_{j_2}\}}\subset S_{\Lambda\setminus\{\alpha_{j_1},\alpha_{j_2}\}}$ acting on $B_{\Lambda\setminus\{\alpha_{j_1},\alpha_{j_2}\}}$, we continue the procedure. This algorithm ends at some point, since $M$ has finite dimension and the dimensions of successive boundary components $B_{\Lambda\setminus\{\alpha_{j_1}\}}\supset B_{\Lambda\setminus\{\alpha_{j_1},\alpha_{j_2}\}} \supset\dots$ form a strictly decreasing sequence. Say that the sequence of boundary components we get is
\[
M=B_{\Phi_0}\supset B_{\Phi_1}\supset B_{\Phi_2} \supset\dots\supset B_{\Phi_m},
\]
where we put $\Phi_0=\Lambda$ and $\Phi_i=\Lambda\setminus\{\alpha_{j_1},\dots,\alpha_{j_i}\}$, for $i=1,\dots, m$, where $m$ must be strictly lower than the rank of $M$ (otherwise $B_{\Phi_m}$ is just one point, and there are no cohomogeneity one actions on it). Thus, our recurrence assumption is that we have a finite sequence of groups
\[
H=H_{\Phi_0}\subset G=S_{\Phi_0}, \quad H_{\Phi_1}\subset S_{\Phi_1}, \quad H_{\Phi_2} \subset S_{\Phi_2}, \quad \dots,\quad  H_{\Phi_m}\subset S_{\Phi_m}
\]
such that each $H_{\Phi_i}$-action on $B_{\Phi_i}$ is orbit equivalent via an element $g_i\in S_{\Phi_i}$ (see Remarks~\ref{rmk:th:BT} and~\ref{rem:scaling_reducible}) to the canonical extension of the $H_{\Phi_{i+1}}$-action on $B_{\Phi_{i+1}}$ to $B_{\Phi_{i}}$, for each $i=0,1,\dots, m-1$,
and the $H_{\Phi_m}$-action on $B_{\Phi_m}$ is no longer orbit equivalent to a canonical extension from any smaller boundary component of $B_{\Phi_m}$.
Since $g_i H_{\Phi_i} g_i^{-1}$ and $H_{\Phi_{i+1}}^{\Phi_i}$ act on $B_{\Phi_i}$ with the same orbits, 
their canonically extended actions of $(g_i H_{\Phi_i} g_i^{-1})^\Lambda$ and $(H_{\Phi_{i+1}}^{\Phi_i})^\Lambda$ on $M$ have exactly the same orbits, by construction. Moreover, the actions of $(g_i H_{\Phi_i} g_i^{-1})^\Lambda$ and $H_{\Phi_i}^\Lambda$ on $M$ are orbit equivalent, because the actions of $H_{\Phi_i}$ and $g_i H_{\Phi_i} g_i^{-1}$ on $B_{\Phi_i}$ are trivially orbit equivalent by the inner isometry $g_i\in S_{\Phi_i}$ of $B_{\Phi_i}$ (see~\cite[Proposition~4.2]{BT:crelle} or the description of the canonical extension in~\S\ref{subsec:classes}). Also, by Lemma~\ref{lemma:canonical}, $(H_{\Phi_{i+1}}^{\Phi_i})^\Lambda=H_{\Phi_{i+1}}^\Lambda$. 
Altogether, we obtain that the actions of $H_{\Phi_i}^\Lambda$ and $H_{\Phi_{i+1}}^\Lambda$ on $M$ are orbit equivalent, for each $i=1,\dots,m-1$. Therefore, the actions of $H_{\Phi_1}^\Lambda$ and $H_{\Phi_m}^\Lambda$ on $M$ are orbit equivalent. Since $g_0 H_{\Phi_0}g_0^{-1}$ and $H^\Lambda_{\Phi_1}$ act on $M$ with the same orbits, we conclude that the action of $H=H_{\Phi_0}$ on $M$ is orbit equivalent to the action of $H_{\Phi_m}^\Lambda$ on $M$.

Now we apply the procedure described at the beginning of the proof (for actions with singular orbits) to the action of $H_{\Phi_m}$ on $B_{\Phi_m}$ instead of the action of $H$ on $M$. Let $\g{s}_{\Phi_m}=\bigoplus_{i=1}^{s_m} (\g{s}_{\Phi_m})_i$ be the decomposition of $\g{s}_{\Phi_m}$ into simple ideals. Since by construction the action of $H_{\Phi_m}$ on $B_{\Phi_m}$ is not a canonical extension, we have one of the following possibilities:
\begin{enumerate}[{\rm (i)}]
\item  The $H_{\Phi_m}$-action on $B_{\Phi_m}$ has the same orbits as the action of a maximal proper reductive subgroup of $S_{\Phi_m}$ with Lie algebra $\bigoplus_{\begin{subarray}{c}i=1\\i\neq j,k\end{subarray}}^{s_m} (\g{s}_{\Phi_m})_i \oplus (\g{s}_{\Phi_m})_{j,k,\sigma_m}$, where $\sigma_m$ is an isomorphism between $(\g{s}_{\Phi_m})_j$ and $(\g{s}_{\Phi_m})_k$, $j,k\in\{1,\dots,s_m\}$, $j\neq k$.
\item  The $H_{\Phi_m}$-action on $B_{\Phi_m}$ has the same orbits as the action of a maximal proper reductive subgroup of $S_{\Phi_m}$ with Lie algebra $\bigoplus_{\begin{subarray}{c}i=1\\i\neq l\end{subarray}}^{s_m} (\g{s}_{\Phi_m})_i \oplus \g{q}_l$, where $\g{q}_l$ is a maximal proper reductive subalgebra of $(\g{s}_{\Phi_m})_l$.
\item The $H_{\Phi_m}$-action on $B_{\Phi_m}$ has the same orbits as the action of $g_m H_{\Phi_m\setminus\{\alpha_k\},\Phi_m,\g{v}} g_m^{-1}$, where $g_m\in S_{\Phi_m}$, $\alpha_k\in \Phi_m$, $\g{v}$ is a subspace of $\g{n}^1_{\Phi_m\setminus\{\alpha_k\},\Phi_m}$, and $ H_{\Phi_m\setminus\{\alpha_k\},\Phi_m,\g{v}}\subset S_{\Phi_m}$ is obtained by nilpotent construction.
\end{enumerate}

In case (i) we must have $s_m=2$, $\Phi_m$ has two elements and $B_{\Phi_m}$ is the product of two symmetric spaces of rank one, because otherwise (by Lemma~\ref{lemma:product}) the $H_{\Phi_m}$-action on $B_{\Phi_m}$ would be orbit equivalent to a canonical extension, which contradicts the definition of $\Phi_m$. This situation corresponds to an action of type (CER) in the statement of Theorem~\ref{th:classification}, where the reducible boundary component of rank two is precisely $B_{\Phi_m}$.

Similarly, in case (ii) we have that $\g{s}_{\Phi_m}$ is a simple Lie algebra for the same reason (and thus $s_m=l=1$), and this corresponds to an action of type (CEI), where the irreducible boundary component is $B_{\Phi_m}$.

Finally, in case (iii), since $g_m\in S_{\Phi_m}$ is an inner isometry of $B_{\Phi_m}$, we have that the canonical extensions of the actions of $H_{\Phi_m}$ and $ H_{\Phi_m\setminus\{\alpha_k\},\Phi_m,\g{v}}$ on $B_{\Phi_m}$ to $M$ are orbit equivalent. As shown above, the $H$-action and the $H_{\Phi_m}^\Lambda$-action on $M$ are orbit equivalent, so we get that the $H$-action is orbit equivalent to the $H_{\Phi_m\setminus\{\alpha_k\},\Phi_m,\g{v}}^\Lambda$-action. But then, Lemma~\ref{lemma:nilpotent} guarantees that the $H_{\Phi_m\setminus\{\alpha_k\},\Phi_m,\g{v}}^\Lambda$-action has the same orbits as the action of the group $H_{\Lambda\setminus\{\alpha_k\},\g{v}}$ obtained by nilpotent construction  from the choice of $\g{v}$ as a subset of $\g{n}_{\Lambda\setminus\{\alpha_k\}}^1$. This corresponds to case (NC) in the statement of Theorem~\ref{th:classification}.
\end{proof}

\section{Applications}\label{sec:applications}
The goal of this section is to prove Theorems~\ref{th:sl_n},~\ref{th:reducible} and~\ref{th:rank1} as applications of Theorem~\ref{th:classification}. This will give us explicit descriptions of the cohomogeneity one actions on the symmetric spaces $\s{SL}_{n+1}(\R)/\s{SO}_{n+1}$, $n\geq 1$ (\S\ref{subsec:SLn}), on the products of rank one spaces (\S\ref{subsec:rank1}), and the structure result for cohomogeneity one actions on reducible spaces (\S\ref{subsec:reducible}).

\subsection{Cohomogeneity one actions on $\s{SL}_{n+1}(\R)/\s{SO}_{n+1}$}\label{subsec:SLn}\hfill

For each integer $n\geq 1$, the symmetric space $\s{SL}_{n+1}(\R)/\s{SO}_{n+1}$ has rank $n$ and its root system is of type $(\s{A}_n)$, which in particular means that $\Sigma^+=\{\sum_{i=j}^k\alpha_i:1\leq j\leq k\leq n\}$ for some set of simple roots $\Lambda=\{\alpha_1,\dots,\alpha_{n}\}$. Moreover, $\g{k}_0=0$, $\g{g}_0=\g{a}$, and for each $\lambda\in \Sigma$, the restricted root space $\g{g}_\lambda$ has dimension one. See~\cite[Example~13.2.1]{BCO:book} for a detailed description.

Note that the case $n=1$ leads to the real hyperbolic plane $\R \s{H}^2$, in which case the classification is classical, whereas the case $n=2$ has been studied in~\cite{BT:crelle}.

\begin{proof}[Proof of Theorem~\ref{th:sl_n}]
	Let us analyze the different cases arising in Theorem~\ref{th:classification}. First, the foliation types (FH) and (FS) in Theorem~\ref{th:classification} correspond directly to cases (FH) and (FS) of Theorem~\ref{th:sl_n}.
	
	Let us focus now on case (CEI) of Theorem~\ref{th:classification}, that is, the cohomogeneity one actions on $M$ that arise as canonical extensions from irreducible boundary components. 
	
	Any connected subset $\Phi$ of simple roots in the Dynkin diagram of $\Lambda$ is of the form $\Phi=\{\alpha_j,\dots,\alpha_k\}$, for some $j,k\in\{1,\dots,n\}$, $j\leq k$. In this case, the boundary component $B_\Phi$ is isometric to the irreducible symmetric space $\s{SL}_{k-j+2}(\R)/\s{SO}_{k-j+2}$. By the description of actions of type (CEI), we have to consider the possible cohomogeneity one actions on $B_\Phi$ with a totally geodesic singular orbit. Such actions are induced by maximal proper reductive subgroups of $\s{SL}_{k-j+2}(\R)$. 
	
	If $j=k$, then $B_\Phi\cong\R \s{H}^2$ admits only one cohomogeneity one action with a totally geodesic orbit, up to orbit equivalence. Such action is the one of the isotropy group at some point of $B_\Phi\cong\R \s{H}^2$, which is given by the action of $K_\Phi^0\cong \s{SO}_2$ on $B_\Phi$ up to orbit equivalence, and has a fixed point as singular orbit. The canonical extension of this action to $M$ leads to the the first row of the table of case (CE) in Theorem~\ref{th:sl_n}.
	
	If $j<k$, we have to consider the cohomogeneity one actions on $B_\Phi\cong \s{SL}_{k-j+2}(\R)/\s{SO}_{k-j+2}$ that have a totally geodesic singular orbit. These were classified in~\cite{BT:tohoku}. According to the classification, these actions are orbit equivalent to the action of a maximal proper reductive subgroup of $\s{SL}_{k-j+2}(\R)$ isomorphic to $\s{SL}_{k-j+1}(\R)\times \R$, or, exceptionally in the case that $B_\Phi$ has rank $3$, i.e., $k=j+2$, to the action of a maximal proper reductive subgroup of $\s{SL}_4(\R)$ isomorphic to $\s{Sp}_2(\R)$. The corresponding totally geodesic singular orbits are isometric to $(\s{SL}_{k-j+1}(\R)/\s{SO}_{k-j+1})\times\R$ or to $\s{Sp}_2(\R)/\s{U}_2\cong \s{SO}^0_{2,3}/\s{SO}_2 \s{SO}_3$, respectively. The canonical extensions of such actions from $B_\Phi$ to $M$ yield the cohomogeneity one actions described in the second and third rows of the table in Theorem~\ref{th:sl_n}.
	
	Now, it is straightforward that the actions of type (CER) in Theorem~\ref{th:classification} give rise to the actions of type (CE) described in the fourth row of the table in Theorem~\ref{th:sl_n}. This is so since any rank two reducible boundary component of $M=\s{SL}_{n+1}(\R)/\s{SO}_{n+1}$ is of the form $B_\Phi\cong\R\s{H}^2\times\R\s{H}^2$, where $\Phi=\{\alpha_j,\alpha_k\}$, $|k-j|\geq 2$, is any disconnected subset of two simple roots in the Dynkin diagram of $\Lambda$.
	
	Finally, we have to determine the actions of type (NC) in Theorem~\ref{th:classification}, that is, those obtained via nilpotent construction. For this, fix $j\in\{1,\dots, n\}$. In our context, we have
\begingroup
\allowdisplaybreaks
\begin{align*}
	\g{n}_{\Lambda\setminus\{\alpha_j\}}&=\g{n}_{\Lambda\setminus\{\alpha_j\}}^1=\bigoplus_{i=1}^j\bigoplus_{l=j}^n\g{g}_{\alpha_i+\dots+\alpha_l}\cong\R^j\otimes(\R^{n-j+1})^*,
	\\
	\g{m}_{\Lambda\setminus\{\alpha_j\}}&=\g{a}^{\Lambda\setminus\alpha_j}\oplus\biggl(\bigoplus_{\alpha\in\Sigma_{\{\alpha_1,\dots,\alpha_{j-1}\}}}\g{g}_\alpha\biggr) \oplus\biggl(\bigoplus_{\alpha\in\Sigma_{\{\alpha_{j+1},\dots,\alpha_n\}}}\g{g}_\alpha\biggr)\cong\g{sl}_j(\R)\oplus\g{sl}_{n-j+1}(\R).
\end{align*}
\endgroup
Moreover, the adjoint Lie algebra representation of $\g{m}_{\Lambda\setminus\{\alpha_j\}}$ (resp.\ $\g{k}_{\Lambda\setminus\{\alpha_j\}}$) on $\g{n}_{\Lambda\setminus\{\alpha_j\}}$ is equivalent to the exterior tensor product representation of $\g{sl}_j(\R)\oplus\g{sl}_{n-j+1}(\R)$ (resp.\ $\g{so}_j\oplus\g{so}_{n-j+1}$) on $\R^j\otimes(\R^{n-j+1})^*$. Let us choose orthonormal bases $\{e_1,\dots,e_j\}$ of $\R^j$ and $\{f^1,\dots,f^{n-j+1}\}$ of $(\R^{n-j+1})^*$ in such a way that $e_i\otimes f^l$ can be regarded as a generator of  $\g{g}_{\alpha_{j-i+1}+\dots+\alpha_{j+l-1}}$, and thus, $\{e_i\otimes f^l:1\leq i\leq j, \, 1\leq l\leq n-j+1\}$ is an orthonormal basis of $\g{n}_{\Lambda\setminus\{\alpha_j\}}\cong\R^j\otimes(\R^{n-j+1})^*$. 

Let us assume, without loss of generality, that $j\leq n-j+1$; the case $j>n-j+1$ is completely analogous due to the symmetry of the Dynkin diagram of $\Lambda$. Since the action of $K_{\Lambda\setminus\{\alpha_j\}}^0$ on $\g{n}_{\Lambda\setminus\{\alpha_j\}}$ is equivalent to the isotropy representation of the symmetric space $\s{SO}^0_{j,n-j+1}/\s{SO}_j\s{SO}_{n-j+1}$, such action is polar with $\Xi=\mathrm{span}\{e_i\otimes f^i:i=1,\dots, j\}$ as a section, that is, $\Xi$ intersects all $K_{\Lambda\setminus\{\alpha_j\}}^0$-orbits and always perpendicularly. Thus, up to conjugation by an element of $K_{\Lambda\setminus\{\alpha_j\}}^0$, we can assume that any nonzero subspace $\g{v}$ of $\g{n}_{\Lambda\setminus\{\alpha_j\}}^1=\g{n}_{\Lambda\setminus\{\alpha_j\}}$ contains a unit vector $v=\sum_{i=1}^j v_i e_i\otimes f^i$, $v_i\in\R$. By the condition (NC2) of the nilpotent construction method, we want $\g{v}$ to be such that $N^0_{K_{\Lambda\setminus\{\alpha_j\}}}(\g{v})$ acts transitively on the unit sphere of $\g{v}$. Thus, $\g{v}$ must admit the orthogonal decomposition $\g{v}=\R v\oplus[N_{\g{k}_{\Lambda\setminus\{\alpha_j\}}}(\g{v}), v]$, where the second addend is perpendicular to $\Xi$ by polarity.

Now we take an element in $\g{m}_{\Lambda\setminus\{\alpha_j\}}$, which we identify with some $A+B\in\g{sl}_j(\R)\oplus\g{sl}_{n-j+1}(\R)$, where $A=(a_{il})_{i,l=1}^j$ and $B=(b_{il})_{i,l=1}^{n-j+1}$. For the sake of convenience, let us define $v_i=0$ for $i>j$. Then 
\[
[A+B, v]=\sum_{i=1}^j\sum_{l=1}^{n-j+1}(a_{il}v_l - b_{il}v_i) e_i\otimes f^l=\sum_{i=1}^j(a_{ii}-b_{ii})v_i e_i\otimes f^i + \sum_{i\neq l} (a_{il}v_l - b_{il}v_i) e_i\otimes f^l.
\]
Note that the first sum after the second equal sign belongs to $\Xi$, whereas the second sum is perpendicular to $\Xi$. Thus, if $A+B\in N_{\g{m}_{\Lambda\setminus\{\alpha_j\}}}(\g{v})$,  the first sum must be proportional to $v$, which implies that there exists $\lambda\in\R$ such that $a_{ii}-b_{ii}=\lambda$ for all $i\in\{1,\dots,j\}$ with $v_i\neq 0$. Hence, if there are at least two indices $i_1,i_2\in\{1,\dots,j\}$ such that $v_{i_1}\neq 0\neq v_{i_2}$, then not every $A+B\in\g{sl}_j(\R)\oplus\g{sl}_{n-j+1}(\R)$, with $A$ and $B$ diagonal, normalizes $\g{v}$. Under the identification $\g{m}_{\Lambda\setminus\{\alpha_j\}}\cong \g{sl}_j(\R)\oplus\g{sl}_{n-j+1}(\R)$, this means that the orthogonal projection of $N_{\g{m}_{\Lambda\setminus\{\alpha_j\}}}(\g{n}_{\Lambda\setminus\{\alpha_j\},\g{v}}) = \theta N_{\g{m}_{\Lambda\setminus\{\alpha_j\}}}(\g{v})$ onto $\g{p}$ does not contain the whole subspace $\g{a}^{\{\alpha_1,\dots,\alpha_{j-1}\}}\oplus\g{a}^{\{\alpha_{j+1},\dots,\alpha_n\}}=\g{a}^{\Lambda\setminus\{\alpha_j\}}$. In this case, the group $N^0_{M_{\Lambda\setminus\{\alpha_j\}}}(\g{n}_{\Lambda\setminus\{\alpha_j\},\g{v}})$ cannot act transitively on the boundary component $B_{\Lambda\setminus\{\alpha_j\}}\cong (\s{SL}_j(\R)/\s{SO}_j)\times (\s{SL}_{n-j-1}(\R)/\s{SO}_{n-j-1})$ since $\g{a}^{\Lambda\setminus\{\alpha_j\}}\subset T_o B_{\Lambda\setminus\{\alpha_j\}}$. This means that condition (NC1) does not hold in this case.

Therefore, we must have $v_i=0$ for all except one $i\in\{1,\dots,j\}$. Again, by conjugating by an element of $K_{\Lambda\setminus\{\alpha_j\}}^0$ if necessary, we can assume that $v=e_1\otimes f^1\in\g{g}_{\alpha_j}$. Now let $S+T\in N_{\g{k}_{\Lambda\setminus\{\alpha_j\}}}(\g{v})$, where $S\in\g{k}_{\{\alpha_1,\dots,\alpha_{j-1}\}}\cong\g{so}_j$ and $T\in\g{k}_{\{\alpha_{j+1},\dots,\alpha_{n}\}}\cong\g{so}_{n-j+1}$. Then the element
\[
[S+T,v]=(Se_1)\otimes f^1 +  e_1 \otimes (Tf^1)
\]
belongs to $\g{v}\ominus\R v$. Assume $Se_1\neq 0\neq Tf^1$. Then there exists
$P\in\s{SO}_j$ mapping the orthogonal set $(e_1,Se_1)$ to the orthogonal set $(e_2,\mu_1 e_1)$, for some $\mu_1\neq 0$; and similarly, there exists $Q\in\s{SO}_{n-j+1}$ sending $(f^1, Tf^1)$ to $(f^1,\mu_2f^2)$, for some $\mu_2\neq 0$. Thus, the element $g=(P,Q)\in K_{\Lambda\setminus\{\alpha_j\}}^0\cong \s{SO}_j\times \s{SO}_{n-j+1}$ satisfies
\begin{equation}\label{eq:AdgSTv}
\Ad(g)[S+T,v]=\mu_1 e_1\otimes f^1+\mu_2 e_2\otimes f^2,\quad \text{with} \;\mu_1\neq 0\neq \mu_2.
\end{equation}
Thus, the subspace $\Ad(g)\g{v}$ of $\g{n}_{\Lambda\setminus\{\alpha_j\}}$ intersects $\Xi$ nontrivially. Also, it satisfies conditions (NC1)-(NC2) because $\g{v}$ does so by assumption. But we have shown in the previous paragraph that no subspace of $\g{n}_{\Lambda\setminus\{\alpha_j\}}$ satisfying (NC1)-(NC2) and intersecting $\Xi$ contains an element such as the one on the right hand side of~\eqref{eq:AdgSTv}. This yields a contradiction, which implies that, for each $S+T\in N_{\g{k}_{\Lambda\setminus\{\alpha_j\}}}(\g{v})$, either $Se_1=0$ or $Tf^1=0$. Since $N_{\g{k}_{\Lambda\setminus\{\alpha_j\}}}(\g{v})$ is a vector space, we actually have either $Se_1=0$ or $Tf^1=0$, for all $S+T\in N_{\g{k}_{\Lambda\setminus\{\alpha_j\}}}(\g{v})$. In other words, $[N_{\g{k}_{\Lambda\setminus\{\alpha_j\}}}(\g{v}),v]\subset \mathrm{span}\{e_1\otimes f^i:i=1,\dots,n-j+1\}$ or $[N_{\g{k}_{\Lambda\setminus\{\alpha_j\}}}(\g{v}),v]\subset \mathrm{span}\{e_i\otimes f^1:i=1,\dots,j\}$. Assume that we are in the second case, that is, $\g{v}=\R v\oplus[N_{\g{k}_{\Lambda\setminus\{\alpha_j\}}}(\g{v}),v]\subset \mathrm{span}\{e_i\otimes f^1:i=1,\dots,j\}$; the first case is completely analogous. Let $k=\dim \g{v}$. Again, up to conjugation by an element of $K^0_{\{\alpha_1,\dots,\alpha_{j-1}\}}\cong \s{SO}_{j}$ we can assume 
\begin{equation}\label{eq:v_sln}
\g{v}=\g{g}_{\alpha_{j-k+1}+\dots+\alpha_j}\oplus\g{g}_{\alpha_{j-k+2}+\dots+\alpha_j}\oplus\dots\oplus\g{g}_{\alpha_{j-1}+\alpha_j}\oplus\g{g}_{\alpha_j}.
\end{equation}
Let $\Omega=\{\alpha_i: 1\leq i\leq j-k-1,\text{ or } j-k+1\leq i\leq j-1, \text{ or } j+2\leq i \leq n\}$.
Then the connected subgroup of $K_{\Lambda\setminus\{\alpha_j\}}$ with Lie algebra
\[
N_{\g{k}_{\Lambda\setminus\{\alpha_j\}}}(\g{v}) = \bigoplus_{\lambda\in\Sigma^+_\Omega}\g{k}_\lambda
\cong \g{so}_{j-k}\oplus\g{so}_k\oplus\g{so}_{n-j}
\]
acts transitively on the unit sphere of $\g{v}$. Moreover, $N_{\g{m}_{\Lambda\setminus\{\alpha_j\}}}(\g{n}_{\Lambda\setminus\{\alpha_j\},\g{v}})$ contains $\g{a}^{\Lambda\setminus\{\alpha_j\}}\oplus\g{n}^{\Lambda\setminus\{\alpha_j\}}$, which is the solvable part of the Iwasawa decomposition of $\g{s}_{\Lambda\setminus\{\alpha_j\}}$, and hence, $N^0_{M_{\Lambda\setminus\{\alpha_j\}}}(\g{n}_{\Lambda\setminus\{\alpha_j\},\g{v}})$ acts transitively on $B_{\Lambda\setminus\{\alpha_j\}}$. Thus, the subspace $\g{v}$ of $\g{n}_{\Lambda\setminus\{\alpha_j\}}^1$ given in~\eqref{eq:v_sln} satisfies both conditions (NC1)-(NC2) of the nilpotent construction method. However, the corresponding cohomogeneity one action on $M$ is orbit equivalent to a canonical extension. Indeed, on the one hand, the Lie algebra of the resulting group $H_{\Lambda\setminus\{\alpha_j\},\g{v}}$ that acts with cohomogeneity one on $M$ satisfies
\[
\g{h}_{\Lambda\setminus\{\alpha_j\},\g{v}}=N_{\g{l}_{\Lambda\setminus\{\alpha_j\}}}(\g{n}_{\Lambda\setminus\{\alpha_j\},\g{v}})\oplus\g{n}_{\Lambda\setminus\{\alpha_j\},\g{v}}\supset\g{a}\oplus\g{n}^{\Lambda\setminus\{\alpha_j\}}\oplus \g{n}_{\Lambda\setminus\{\alpha_j\},\g{v}}=\g{a}\oplus(\g{n}\ominus\g{v}).
\]
By dimension reasons, the singular orbit of the $H_{\Lambda\setminus\{\alpha_j\},\g{v}}$-action is also an orbit of the connected Lie subgroup of $AN$ with Lie algebra $\g{a}\oplus(\g{n}\ominus\g{v})$. On the other hand, let $\Psi=\{\alpha_{j-k+1},\dots,\alpha_{j-1}\}\subset\{\alpha_{j-k+1},\dots,\alpha_{j}\}=\Phi$. Consider the boundary component $B_{\Phi}\cong \s{SL}_{k+1}(\R)/\s{SO}_{k+1}$ and the cohomogeneity one action on $B_\Phi$ of the reductive subgroup $L_{\Psi,\Phi}^0\cong\s{SL}_k(\R)\times \R$ of $\s{SL}_{k+1}(\R)$ with Lie algebra $\g{l}_{\Psi,\Phi}$ (recall the notation at the end of~\S\ref{subsec:parabolic}). Then the Lie algebra $\g{l}_{\Psi,\Phi}^\Lambda$ of the group obtained by canonical extension of the action of $L_{\Psi,\Phi}^0\cong\s{SL}_k(\R)\times \R$ on $B_\Phi$ to $M$ has the following projection onto $\g{a}\oplus\g{n}$:
\[
(\g{l}_{\Psi,\Phi}^\Lambda)_{\g{a}\oplus\g{n}}=(\g{l}_{\Psi,\Phi}\oplus\g{a}_\Phi\oplus\g{n}_\Phi)_{\g{a}\oplus\g{n}}= \g{a}^\Phi\oplus\g{n}^\Psi\oplus\g{a}_\Phi\oplus\g{n}_\Phi
=\g{a}^\Phi\oplus(\g{n}^\Phi\ominus\g{v})\oplus\g{a}_\Phi\oplus\g{n}_\Phi=\g{a}\oplus(\g{n}\ominus\g{v}),
\]
where we have used $\g{n}^\Phi\ominus\g{n}^\Psi=\g{n}_{\Psi,\Phi}=\g{g}_{\alpha_{j-k+1}+\dots+\alpha_j}\oplus\dots\oplus\g{g}_{\alpha_j}=\g{v}$. By dimension reasons, the singular orbits of the cohomogeneity one actions of $H_{\Lambda\setminus\{\alpha_j\},\g{v}}$ and $(L_{\Psi,\Phi}^0)^\Lambda$ on $M$ coincide. Hence, both actions have the same orbits. We conclude that the action of $H_{\Lambda\setminus\{\alpha_j\},\g{v}}$ is orbit equivalent to one of the actions in the second row of the table of item (CE) in Theorem~\ref{th:sl_n}, namely the canonical extension of the action of $\s{SL}_{k}(\R)\times \R$ on~$B_{\{\alpha_{j-k+1},\dots,\alpha_{j}\}}$~to~$M$.
\end{proof}

\subsection{Cohomogeneity one actions on reducible symmetric spaces}\label{subsec:reducible}\hfill

Consider a symmetric space of noncompact type $M=M_1\times\dots\times M_s=G/K$, where each $M_i=G_i/K_i$, $i=1,\dots,s$, is irreducible, $G=\prod_{i=1}^s G_i$, and $K=\prod_{i=1}^s K_i$. Clearly, the root system of $\g{g}=\bigoplus_{i=1}^s\g{g}_i$ splits as the orthogonal disjoint union $\Sigma=\bigsqcup_{i=1}^s\Sigma_i$, where $\Sigma_i$ is the root system of $\g{g}_i$. Similarly, a set of simple roots for $\g{g}$ is given by $\Lambda=\bigsqcup_{i=1}^s\Lambda_s$, where $\Lambda_i\subset\Sigma_i^+$ is a set of simple roots for $\g{g}_i$, $i=1,\dots,s$. We will denote by $\g{k}_i\oplus\g{a}_i\oplus\g{n}_i$ the associated Iwasawa decomposition of $\g{g}_i$. Observe that $\g{a}_i=\g{a}^{\Lambda_i}$ and $\g{n}_i=\g{n}^{\Lambda_i}$, $i=1,\dots,s$. Of course, we have orthogonal direct sums $\g{k}=\bigoplus_{i=1}^s\g{k}_i$, $\g{a}=\bigoplus_{i=1}^s\g{a}_i$ and $\g{n}=\bigoplus_{i=1}^s\g{n}_i$.

\begin{proof}[Proof of Theorem~\ref{th:reducible}]
	We have to analyze the different cases arising in Theorem~\ref{th:classification} when applied to a reducible $M$. First note that cases (FH) and (CER) in Theorem~\ref{th:classification} correspond directly to cases (FH) and (CER) in Theorem~\ref{th:rank1}, respectively.
	
	An action of type (FS) in Theorem~\ref{th:classification} is induced by the Lie algebra $\g{a}\oplus(\g{n}\ominus\ell)$, where $\ell$ is a subspace of a simple root space $\g{g}_\beta$, $\beta\in\Lambda$, with $\dim\ell=1$. Then $\beta\in\Lambda_j$, for some $j\in\{1,\dots, s\}$, and hence 
	\[
	\g{a}\oplus(\g{n}\ominus\ell)=\g{a}\oplus\Bigl(\bigoplus_{\begin{subarray}{c}i=1\\i\neq j\end{subarray}}^s \g{n}_i\Bigr)\oplus(\g{n}_j\ominus\ell)
	=
	\Bigl(\bigoplus_{\begin{subarray}{c}i=1\\i\neq j\end{subarray}}^s (\g{a}_i\oplus\g{n}_i)\Bigr)\oplus(\g{a}_j\oplus(\g{n}_j\ominus\ell))
	=\g{a}_{\Lambda_j}\oplus\g{n}_{\Lambda_j}\oplus\g{h}_j,
	\]
	where $\g{h}_j=\g{a}_j\oplus(\g{n}_j\ominus\ell)$. Then the corresponding action is the canonical extension of the $H_j$-action on $B_{\Lambda_j}\cong M_j$ to $M$, where $H_j$ is the connected subgroup of $G_j$ with Lie algebra $\g{h}_j$. By Lemma~\ref{lemma:product}, such action is orbit equivalent to the action of $H_j\times \prod_{\begin{subarray}{c}i=1\\i\neq j\end{subarray}}^s G_i$, which corresponds to case (Prod) in the statement of Theorem~\ref{th:reducible}.
	
	Case (CEI) of Theorem~\ref{th:classification} concerns canonical extensions of cohomogeneity one actions with a totally geodesic singular orbit on an irreducible boundary component $B_\Phi$, for some connected subset $\Phi$ of roots in the Dynkin diagram. The corresponding Lie algebras are of the form $\g{h}_\Phi^\Lambda=\g{h}_\Phi\oplus\g{a}_\Phi\oplus\g{n}_\Phi$, for some maximal proper reductive subalgebra $\g{h}_\Phi$ of $\g{s}_\Phi$ whose corresponding Lie subgroup of $S_\Phi$ acts with cohomogeneity one on $B_\Phi$. In our setting, being $B_\Phi$ irreducible implies $\Phi\subset\Lambda_j$, for some $j\in\{1,\dots,s\}$. Hence
	\[
	\g{h}_\Phi\oplus\g{a}_\Phi\oplus\g{n}_\Phi=\g{h}_\Phi\oplus(\g{a}_{\Phi,\Lambda_j}\oplus\g{a}_{\Lambda_j})\oplus(\g{n}_{\Phi,\Lambda_j}\oplus\g{n}_{\Lambda_j})=(\g{h}_\Phi\oplus\g{a}_{\Phi,\Lambda_j}\oplus\g{n}_{\Phi,\Lambda_j})\oplus\g{a}_{\Lambda_j}\oplus\g{n}_{\Lambda_j},
	\]
	which means that the action is a composition of canonical extensions, firstly from $B_\Phi$ to $B_{\Lambda_j}\cong M_j$, and secondly from $B_{\Lambda_j}\cong M_j$ to $M$. Again by Lemma~\ref{lemma:product} we get that the action of the connected subgroup of $G$ with Lie algebra $\g{h}_\Phi^\Lambda$ has the same orbits as the action of $H_j\times \prod_{\begin{subarray}{c}i=1\\i\neq j\end{subarray}}^s G_i$ on $M$, where $H_j$ is the connected subgroup of $G$ with Lie algebra $\g{h}_\Phi\oplus\g{a}_{\Phi,\Lambda_j}\oplus\g{n}_{\Phi,\Lambda_j}$. This fits again into case (Prod) in the statement. 
	
	Finally, case (NC) of Theorem~\ref{th:classification} describes a nilpotent construction from a subspace $\g{v}$ of $\g{n}^1_{\Lambda\setminus\{\beta\}}$ for some $\beta\in \Lambda$, $\dim\g{v}\geq 2$. Let $j\in\{1,\dots, s\}$ such that $\beta\in\Lambda_j$. Then 
	\[
	\g{n}_{\Lambda\setminus\{\beta\}}^1 = \g{n}_{\Lambda\setminus\{\beta\},\Lambda_j}^1 \subset \g{n}_{\Lambda\setminus\{\beta\}} = \g{n}_{\Lambda\setminus\{\beta\},\Lambda_j} \subset \g{n}^{\Lambda_j}=\g{n}_j,
	\]
	since any root not spanned by $\Lambda\setminus\{\beta\}$ must be spanned by roots in $\Lambda_j$. Note that 
	$\g{l}_{\Lambda\setminus\{\beta\}}
	=\big(\bigoplus_{\begin{subarray}{c}i=1\\i\neq j\end{subarray}}^s\g{g}_i\bigr)\oplus\g{l}_{\Lambda_j\setminus\{\beta\},\Lambda_j}$.
	Hence the Lie algebra of the group $H_{\Lambda_j\setminus\{\beta\},\g{v}}$ built by nilpotent construction from the choice $\g{v}\subset \g{n}_{\Lambda\setminus\{\beta\}}^1=\g{n}_{\Lambda\setminus\{\beta\},\Lambda_j}^1$ is
	\begin{equation}\label{eq:productNC}
	N_{\g{l}_{\Lambda\setminus\{\beta\}}}(\g{n}_{\Lambda\setminus\{\beta\}}\ominus\g{v})\oplus(\g{n}_{\Lambda\setminus\{\beta\}}\ominus\g{v})
=
	\big(\bigoplus_{\begin{subarray}{c}i=1\\i\neq j\end{subarray}}^s\g{g}_i\bigr)
	\oplus N_{\g{l}_{\Lambda_j\setminus\{\beta\},\Lambda_j}}(\g{n}_{\Lambda_j\setminus\{\beta\},\Lambda_j}\ominus\g{v})\oplus(\g{n}_{\Lambda_j\setminus\{\beta\},\Lambda_j}\ominus\g{v}),
	\end{equation}
	where the two last direct addends of the right-hand term constitute a Lie subalgebra of $\g{g}_j$. (Indeed, it is not difficult to show that the associated connected subgroup of $G_j$ yields the cohomogeneity one action on $M_j$ obtained by nilpotent construction from the choice of $\g{v}$ as a subspace of $\g{n}_{\Lambda\setminus\{\beta\},\Lambda_j}^1$.) We conclude that the group $H_{\Lambda_j\setminus\{\beta\},\g{v}}$  splits nicely with respect to the decomposition of $G$, and hence corresponds again to an action of type (Prod).
	\end{proof}

\subsection{Cohomogeneity one actions on products of rank one spaces}\label{subsec:rank1}\hfill

In this subsection, $M=M_1\times\dots\times M_r$ will be a product of symmetric spaces of noncompact type and rank one, $M_i=G_i/K_i=\mathbb{F}_i \s{H}^{n_i}$, where $\mathbb{F}_i\in\{\R,\C,\H,\mathbb{O}\}$, $i=1,\dots, r$. We will use the other notations stated at the beginning of~\S\ref{subsec:reducible}. 
Moreover, we have $\Lambda=\{\alpha_1,\dots,\alpha_r\}$ with $\Lambda_i=\{\alpha_i\}$, $\Sigma^+=\Lambda\cup\{2\alpha_i:\mathbb{F}_i\neq\R\}$, and $\g{a}_i\cong\R$, for each $i=1,\dots, r$.

\begin{proof}[Proof of Theorem~\ref{th:rank1}]
	We will go through the three types of actions in Theorem~\ref{th:reducible}.
	
	First, assume we have an action of (Prod) type, that is, the action of a subgroup $H=H_j\times\prod_{\begin{subarray}{c}i=1\\i\neq j\end{subarray}}^r G_i$ of $G$, where $H_j$ is a  connected subgroup of $G_j$ acting with cohomogeneity one on the rank one space $M_j=B_{\{\alpha_j\}}=\mathbb{F}_j\s{H}^{n_j}$. By the classification of cohomogeneity one actions on rank one spaces~\cite{BT:tams}, \cite{DDR:crelle}, we can distiguish four cases:
	\begin{enumerate}
		\item $H_j$ produces a foliation of horospherical type. In this case, up to orbit equivalence, we can assume $\g{h}_j=\g{n}_j$ (since $\dim\g{a}_j=1$), and it is easy to realize that $H$ induces a foliation of horospherical type on $M$, with the same orbits as the action of the connected subgroup of $G$ with Lie algebra $(\g{a}\ominus\g{a}_j)\oplus\g{n}$. This corresponds to item (FH) in the statement.
		\item $H_j$ produces a foliation of solvable type. In this case we can assume $\g{h}_j=\g{a}_j\oplus(\g{n}_j\ominus\ell)$, for some one-dimensional subspace $\ell$ of $\g{g}_{\alpha_j}$. Similarly as above, one can see that the $H$-action is orbit equivalent to the action described in item (FS) of the statement.
		\item $H_j$ acts with cohomogeneity one and a totally geodesic singular orbit on $M_j$, which translates directly into type (CEI) of the statement.
		\item $H_j$ acts with cohomogeneity one and a non-totally geodesic singular orbit on $M_j$. In this case, $\g{h}_j$ can be taken of the form $N_{(\g{k}_j)_0}(\g{v})\oplus\g{a}_j\oplus(\g{n}_j\ominus\g{v})$, for some protohomogeneous subspace $\g{v}$ of $\g{g}_{\alpha_j}$ with $\dim\g{v}\geq 2$. This yields an action of type (NC) in the statement.
	\end{enumerate}

	Now, clearly an action of (FH) type in Theorem~\ref{th:reducible} corresponds to an action of the same type in Theorem~\ref{th:rank1}.
	
	Finally, actions of type (CER) are induced by groups $H_\Phi^\Lambda$ with Lie algebras of the type $\g{h}_\Phi\oplus\g{a}_\Phi\oplus\g{n}_\Phi$, where $\Phi\subset\Lambda$ determines a reducible rank two boundary component $B_\Phi$, which, in the current context, is of the form $B_\Phi=M_j\times M_k$, for $\Phi=\{\alpha_j,\alpha_k\}$,  $j,k\in\{1,\dots,r\}$, $j\neq k$. Hence, $\g{s}_{\{\alpha_j\}}=\g{g}_j$ and $\g{s}_{\{\alpha_k\}}=\g{g}_k$, so the Lie algebra of the group acting diagonally on $B_\Phi$ is $\g{h}_\Phi=\{X+\sigma X: X\in\g{g}_j\}=\g{g}_{j,k,\sigma}$, for some Lie algebra isomorphism $\sigma\colon\g{g}_j\to \g{g}_k$. Since $\Phi$ and $\Lambda\setminus\Phi$ are the sets of simple roots associated with $\g{g}_j\oplus\g{g}_k$ and $\bigoplus_{\begin{subarray}{c}i=1\\i\neq j,k\end{subarray}}\g{g}_i$, respectively, we can apply Lemma~\ref{lemma:product} to conclude that the $H_\Phi^\Lambda$-action is orbit equivalent to the action of the connected subgroup of $G$ with Lie algebra $\bigoplus_{\begin{subarray}{c}i=1\\i\neq j,k\end{subarray}}\g{g}_i\oplus\g{g}_{j,k,\sigma}$, as in item (CER) of the statement.
\end{proof}

%%%%%%%%%%%%%%%%%%%%%%%%%% Bibliography %%%%%%%%%%%%%%%%%%%%%%%%%%%
%\bibliographystyle{amsplain}

\end{document}